\documentclass[reqno,final]{amsart}
\usepackage{natbib}  
\usepackage{fancyhdr} 
\usepackage{color} 
\usepackage{hyperref} 
\usepackage{graphicx} 

\usepackage{amsfonts}
\usepackage{amsmath}
\usepackage{mathrsfs, enumerate, amssymb}
\usepackage{tikz}\usepackage{fancybox}
\usepackage{pstricks}
\usepackage{pst-node}
\usepackage[headinclude,DIV13]{typearea}

\usepackage{geometry}
\usepackage{psfrag}

\usepackage{pstricks}
\usepackage{amssymb}

\definecolor{aleacolor}{rgb}{0.16,0.59,0.78}

\hypersetup{
	breaklinks,
	colorlinks=true,
	linkcolor=aleacolor,
	urlcolor=aleacolor,
	citecolor=aleacolor}

\theoremstyle{plain}
\newtheorem{theorem}{Theorem}[section]
\newtheorem{proposition}[theorem]{Proposition}
\newtheorem{lemma}[theorem]{Lemma}

\theoremstyle{definition}

\theoremstyle{remark}
\newtheorem{remark}[theorem]{Remark}

\makeatletter \@addtoreset{equation}{section} \makeatother
\renewcommand\theequation{\thesection.\arabic{equation}}
\renewcommand\thefigure{\thesection.\arabic{figure}}
\renewcommand\thetable{\thesection.\arabic{table}}

\renewcommand{\cite}{\citet}

\newcommand{\p}{\mathbb{P}}
\newcommand{\e}{\mathbb{E}}
\newcommand{\ud}{\mathrm{d}}

\newcommand{\R}{\mathbb{R}}
\renewcommand{\P}{\mathbb{P}}
\newcommand{\N}{\mathbb{N}}
\newcommand{\E}{\mathbb{E}}

\newcommand{\Indi}[1]{\mathbf{1}_{\{#1\}}}
\newcommand{\Ind}{\mathbf{1}}
\newcommand{\Expx}[2]{\mathbb{E}_{#1}\left[ #2 \right]}
\newcommand{\Prox}[2]{\mathbb{P}_{#1}\left( #2 \right)}
\newcommand{\Expconx}[3]{\mathbb{E}_{#1}\left[#2\big| #3 \right]}
\newcommand{\bra}[2]{\langle #1,#2\rangle} 

\newcommand{\mbf}[1]{\boldsymbol{#1}} 

\begin{document}

\title[Almost sure growth of supercritical multi-type continuous-state branching process]
{Almost sure growth of supercritical multi-type continuous-state branching process}
\thanks{}
\author{Andreas E. Kyprianou}
\address{Department of Mathematical Sciences, University of Bath, Claverton Down, Bath, BA2 7AY, UK.}
\email{a.kyprianou@bath.ac.uk}

\author{Sandra Palau}
\address{Department of Mathematical Sciences, University of Bath, Claverton Down, Bath, BA2 7AY, UK.}
\email{sp2236@bath.ac.uk}

\author{Yan-Xia Ren}
\address{LMAM School of Mathematical Sciences \& Center for Statistical Science, Peking University, Beijing, 100871, P.R. China}
\email{yxren@math.pku.edu.cn}


		\begin{abstract}
		In \cite[Example 2.2]{Li}, the notion of a  multi-type continuous-state branching process  (MCSBP) was introduced with a finite number of types, with the countably infinite case being proposed in \cite{kp}. One may consider such processes as a super-Markov chain on a countable state-space of types, which undertakes both local and non-local branching. In \cite{kp} it was shown that, for  MCSBPs, under mild conditions, there exists a lead eigenvalue which characterises the spectral radius of the  linear semigroup associated to the process. Moreover, in a qualitative sense, the sign of this eigenvalue distinguishes between the cases where there is local extinction and exponential growth. In this paper, we continue in this vein and show that, when the number of types is finite, the lead eigenvalue gives the precise almost sure rate of growth of each type. This result matches perfectly classical analogues for multi-type Galton--Watson processes.	
\end{abstract}

		\maketitle

\section{Introduction}

Let $d\in \N$ be a natural number and put $E=\{1,\cdots,d\}$.
For $\mbf{x}=(x_1,\cdots,x_d)^T\in\R^d$ we use the notation $x(i):=x_i$ for $i\in E$ and denote by $\|\mbf{x}\|$ the Euclidean norm. We consider a multi-type continuous state branching process  with $d$ types with branching mechanism $\psi$, henceforth referred to as a $\psi$-MCSBP. This is to say, we are interested in  a $[0,\infty)^{d}$-valued strong Markov process $\mbf{X}:=(\mbf{X}_t,t\geq 0)$,  with probabilities $\{\p_{\mbf{x}}, \mbf{x}\in \R^{d}_{+}\}$ such that the following branching property hold:
for all $\mbf{x},\mbf{y}\in\R_+^d$,
	$$\Expx{\mbf{x}+\mbf{y}}{ {\rm e}^ {-\langle\mbf{f}, \mbf{X}_t\rangle}}=\Expx{\mbf{x}}{ {\rm e}^ {-\langle\mbf{f}, \mbf{X}_t\rangle}}\Expx{\mbf{y}}{ {\rm e}^ {-\langle\mbf{f}, \mbf{X}_t\rangle}},\qquad t\geq 0, \quad\mbf{f}\in \R_+^d.$$
Its branching mechanism, a vectorial function $\psi:E\times \R_+^d\rightarrow 	\R^d$ such that
	\begin{equation*}\label{branching mechanism}
\psi_i(\mbf{u}):=-\langle \mbf{u},\mbf{\widetilde{B}e}_i\rangle+c_i u_i^2+\int_{\R^d_+}( {\rm e}^ {-\langle\mbf{u},\mbf{z}\rangle}-1+u_i(z_i\wedge 1))\mu_i(\ud \mbf{z}),\qquad \mbf{u}\in\R_+^d, i\in E,	
\end{equation*}
	where $c_i\in \R_+$, $\mbf{\widetilde{B}}$ is a $d\times d$ matrix such that $\widetilde{B}_{i,j}\Indi{i\neq j}\in\R_+$, $\mbf{e_1},\cdots,\mbf{e_d}$ is the natural basis in $\R^d$, and $\mu_i$ is a measure concentrated on $\R_+^d\setminus \{\mbf0\}$ such that
$$\int_{\R_+^d}\left[\|\mbf{z}\|\wedge \|\mbf{z}\|^2+\underset{j\in E}{\sum}\Indi{j\neq i}z_j\right]\mu_i(\ud \mbf{z})<\infty.$$

The process $\mbf{X}$ is characterized by its Laplace transform:
	\begin{equation}
	\label{laplace}
\Expx{\mbf{x}}{ {\rm e}^ {-\bra{\mbf{f}}{\mbf{X}_t}}}= {\rm e}^ {-\bra{\mbf{x}}{\mbf{v}(t,\mbf{f})}}, \qquad \mbf{x},\mbf{f}\in\R_+^d,\ t\in\R_+,
	\end{equation}
		where, for any $\mbf{f}\in \R_+^d$, the continuous differentiable function
$$t\mapsto \mbf{v}(t,\mbf{f})=(v_1(t,\mbf{f}),\cdots,v_d(t,\mbf{f}))^T,$$ is the unique locally bounded non-negative solution to the system of integral equations
		\begin{equation}
		\label{differential}
		v_i(t,\mbf{f})=f_i-\int_0^t\psi_i(\mbf{v}(s,\mbf{f}))\ud s, \qquad  \qquad i\in E.
		\end{equation}
	According to \cite{blp1}, this process can be seen as a strong solution of a stochastic differential equation (SDE). More precisely, let  $\mbf{W}_t$ be a $d$-dimensional standard Brownian motion, and for each $i\in E$, let $N_i$ be a Poisson random measure on $\R_+\times\R_+^d\times \R_+$ with intensity measure $\ud s\mu_i(\ud \mbf{z})\ud r$, and denote by $\widetilde{N}_i$ its compensated measure. Suppose that $\mbf{W}$ and $(N_i)_{i\in E}$ are independent of each other. Then, a MCSBP with branching mechanism $\psi$ is characterized as the unique $\R_+^d$-valued strong solution to the SDE

\begin{equation}\label{sde}\begin{split}
\mbf{X}_t=&\mbf{X}_0+\int_0^t \mbf{BX}_s\ud s+\underset{i\in E}{\sum} \mbf{e}_i\int_0^t\sqrt{2c_iX_{s,i}}\ud W_{s,i}+\underset{i\in E}{\sum}\int_0^t\int_{\R_+^d}\int_0^{\infty} \mbf{z}\Indi{r\leq X_{s-,i}}\widetilde{N}_i(\ud s,\ud \mbf{z}, \ud r),
\end{split}
\end{equation}
where the matrix $\mbf{B}$ is given by
$$B_{i,j}=\widetilde{B}_{i,j}+\int_{\R_+^d}(z_i-\delta_{i,j})^+\mu_j(\ud \mbf{z}).$$
Moreover, they proved in \cite[Formula (2.15) and the later computations]{blp1} that $\psi$ can be written as
\begin{equation}
\label{def psi}
\psi_i(\mbf{u})=-\langle \mbf{u},\mbf{{B}e}_i\rangle+c_i u_i^2+\int_{\R^d}( {\rm e}^ {-\langle\mbf{u},\mbf{z}\rangle}-1+\bra{\mbf{u}}{\mbf{z}})\mu_i(\ud \mbf{z}),\qquad \mbf{u}\in\R_+^d,\ i\in E.
\end{equation}
\begin{remark}\rm
Thanks to the  relation between $\mbf{B}$ and $\widetilde{\mbf{B}}$, when we write the branching me\-cha\-nism as in \eqref{def psi}, the matrix $\mbf{B}$ satisfies
\begin{equation}
\label{inequality B}
\int_{\R^d_+}{z}(i) {\mu_{j}(\ud \mbf{z})}\leq B_{ij} \qquad \mbox{ for all } i \neq j.
\end{equation}
\end{remark}

\begin{remark}\rm
If we regard $E$ as the space where particles located, the model we described above can be seen as a special case of a superprocess in which the associated Markov movement is that of a Markov chain on $E$. Indeed, for such a process, from e.g. \cite{dyn91},  the log-Laplace semigroup, $\mbf{V}_t
:\mathbb{R}_+^d\mapsto\mathbb{R}_+^d
$, which similarly to \eqref{laplace}, satisfies $\langle\mbf{x}, \mbf{V}_t\mbf{f}\rangle = -\log \Expx{\mbf{x}}{ {\rm e}^ {-\bra{\mbf{f}}{\mbf{X}_t}}}$, for $\mbf{f}, \mbf{x}\in\mathbb{R}^d_+$, and is the unique solution
to
\begin{equation*}
\mbf{V}_t\mbf{f}(i)=[ {\rm e}^ {t\mbf{Q}}\mbf{f}](i)-\int_0^t {\rm e}^ {(t-s)\mbf{Q}}\psi_{i}(V_{s}\mbf{f})\ud s
\end{equation*}
where
$\mbf{f} = (f(1),\cdots, f(d))^T\in\R_+^d$
and $\mbf{Q}$ is the infinitesimal generator of the associated Markov chain. Note that a straightforward manipulation in the spirit of Theorem 3.1.2 of  \cite{dynkin}
implies that
\begin{equation*}
\mbf{V}_t\mbf{f}(i)=f(i)+\int_0^t[\mbf{Q}V_{s}\mbf{f}](i)-\int_0^t\psi_{i}(V_{s}\mbf{f})\ud s=f(i)+\int_0^t\langle V_{s}\mbf{f},\mbf{Q}^T\mbf{e}_i\rangle-\int_0^t\psi_{i}(V_{s}\mbf{f})\ud s,\end{equation*}
where $\mbf{Q}^T$ is the transpose matrix of $\mbf{Q}$.
In turn, we note  that this is equivalent to the unique semigroup evolution that solves \eqref{differential}, albeit that the branching mechanism
\begin{equation*}
\widetilde{\psi}_i(\mbf{u}):=-\langle \mbf{u},\mbf{Q}^T\mbf{e}_i\rangle+\psi_i(\mbf{u})\qquad \mbf{u}\in\R_+^d,\ i\in E.
	\end{equation*}
\end{remark}

We therefore suppose in the remaining part  of this paper that $\{\mbf{X}_t, t\ge 0\}$ is  a ${\psi}$-MCSBP.
Denote by $\mbf{M}(t):=(M(t)_{i,j})_{d\times d}$ the matrix with elements
$$M(t)_{i,j}:=\Expx{\mbf{e}_i}{\bra{\mbf{e}_j}{\mbf{X}_t}},\qquad i,j\in E,\ t\geq 0.$$
 By \cite[Lemma 3.4]{blp1} we have
\begin{equation}
\mbf{M}(t)= {\rm e}^ {t\mbf{B}^T},\qquad t\geq 0,
\label{eB}
\end{equation}
where $\mbf{B}^T$ is the transpose of $\mbf{B}$. Observe that for any initial vector
 $\mbf{x}_0$ and any  $\mbf{x}\in \R^d_+$,
 \begin{equation*}
 \label{form}
 \Expx{\mbf{x}_0}{\bra{\mbf{x}}{\mbf{X}_t}}=\mbf{x}_0^T\mbf{M}(t)\mbf{x}, \qquad t\geq 0.
 \end{equation*}
Moreover, by \eqref{sde} and the It\^{o} calculus, for all $\mbf{x}\in\R^d$ and  $s\leq r\leq t$, we obtain
\begin{equation}\label{completly useful equation}\begin{split}
\bra{\mbf{M}(t-r)\mbf{x}}{\mbf{X}_r}
=&\bra{\mbf{M}(t-s)\mbf{x}}{\mbf{X}_s}+\underset{i\in E}{\sum} \int_s^r[\mbf{M}(t-u)\mbf{x}]_i\sqrt{2c_iX_{u,i}}\ud W_{u,i}\\
&\hspace{1.5cm}+\underset{i\in E}{\sum}\int_s^r\int_{\R_+^d}\int_0^{\infty} \bra{\mbf{M}(t-u)\mbf{x}}{\mbf{z}}\Indi{l\leq X_{u-,i}}\widetilde{N}_i(\ud u,\ud \mbf{z}, \ud l).
\end{split}
\end{equation}
In particular, if we take $r=t$, we get the following equation
\begin{equation}\label{useful equation}\begin{split}
\bra{\mbf{x}}{\mbf{X}_t}
=&\bra{\mbf{M}(t-s)\mbf{x}}{\mbf{X}_s}+\underset{i\in E}{\sum} \int_s^t[\mbf{M}(t-u)\mbf{x}]_i\sqrt{2c_iX_{u,i}}\ud W_{u,i}\\
&\hspace{1.5cm}+\underset{i\in E}{\sum}\int_s^t\int_{\R_+^d}\int_0^{\infty} \bra{\mbf{M}(t-u)\mbf{x}}{\mbf{z}}\Indi{l\leq X_{u-,i}}\widetilde{N}_i(\ud u,\ud \mbf{z}, \ud l).
\end{split}
\end{equation}

A vector $\mbf{u}\in\R^d$ is called a $\lambda$-right (resp. left) eigenvector if for all $t\geq 0$,
$$\mbf{M}(t)\mbf{u}= {\rm e}^ {\lambda t}\mbf{u}, \qquad (\mbox{resp. }\mbf{u}^T\mbf{M}(t)= {\rm e}^ {\lambda t}\mbf{u}^T).$$
If $\mbf{u}$ is a $\lambda$-right eigenvector, define
$$W_t^{\lambda}(\mbf{u}):= {\rm e}^ {-\lambda t} \bra{\mbf{u}}{\mbf{X}_t}, \qquad t\geq 0.$$
Then it was shown in Proposition 3 of \cite{kp} that if $\mbf{u}$ is a $\lambda$-right eigenvector with $\lambda\in \R$, then for any $
\mbf{x}\in\R^d_+$,  $\{W^\lambda_t(\mbf{u}), t\ge 0\}$ is a martingale under $\p_{\mbf{x}}$.

Suppose that $\mbf{M}(t)$ is irreducible (exists $t_0>0$ such that $M(t_0)_{i,j}>0$ for all $i,j\in E $). The Perron-Frobenius theory implies that there exist $\lambda_1\in\R$ and right and left  associated eigenvectors $\mbf{\phi}, \mbf{\widehat{\phi}}\in\R_+^d$ with all coordinates strictly positive
such that $\mbf{M}(t)\mbf\phi = {\rm e}^{\lambda_1 t}\mbf\phi$ and $\widehat{\mbf\phi}^{T}\mbf{M}(t)  = {\rm e}^{\lambda_1 t}\widehat{\mbf\phi}^T$, for all $t\geq 0$. We note from \eqref{eB} that this is equivalent to the statement that $\mbf\phi$ and $\widehat{\mbf\phi}$ are right and left eigenvectors for $\mbf{B}^{T}$ with common eigenvalue $\lambda_1$.
For convenience we shall normalise $\mbf\phi$ and $\widehat{\mbf\phi}$ such that  $\bra{\mbf{\phi}}{\mbf{1}}=1=\bra{\mbf{\phi}}{\mbf{\widehat{\phi}}}$. Moreover, any other eigenvalue $\lambda$ satisfies $\lambda_1>\Re(\lambda)$ and
\begin{equation*}
\underset{t\rightarrow\infty}{\lim}\mbf{M}(t) {\rm e}^ {-\lambda_1 t}=\mbf{P}:=(\mbf{\phi}_i\mbf{\widehat{\phi}}_j)_{i,j\in E\times E}.
\end{equation*}
In addition, \cite[Lemma A.3]{bp2005} proved that there exist $c_1,c_2,c_3>0$ such that
\begin{equation}
\label{limit M barczy}
\|\mbf{M}(t) {\rm e}^ {-\lambda_1 t}-\mbf{P}\|\leq c_1 {\rm e}^ {-c_2t} \qquad \mbox{ and }\qquad \|\mbf{M}(t)\|\leq c_3 {\rm e}^ {\lambda_1 t} \qquad \mbox{ for all } t\in \R_+.
\end{equation}
In order to simplify notation, we will denote by
\begin{equation}
\label{previousmg}W_t:=W_t^{\lambda_1}(\mbf\phi)= {\rm e}^ {-\lambda_1 t} \bra{\mbf{\phi}}{\mbf{X}_t}, \qquad t\geq 0.
\end{equation} Observe that $W_t$ is a non-negative martingale and has a limit a.s. that we will denote by $W_{\infty}$.

{We say that the $\psi$-MCSBP is {\it subcritical}, {\it critical}, or {\it supercritical} according as $\lambda_1<0$, $\lambda_1=0$, or $\lambda_1>0$. This classification is consistent with the corresponding classification for single-type continuous state branching processes, see, e.g., \cite[Page 58]{Li}.  With the use of \eqref{previousmg}, it was proved  in \cite{kp}, that in the subcritical and critical cases the process has extinction a.s. In this paper, we want to find the asymptotic behaviour in the supercritical case.}

The following Theorem \ref{xlogx theorem} gives
a relationship between the $\mathbb{L}^1$-convergence of the martingale $\{W_t, t\ge 0\}$ and the following condition:
\begin{equation}\tag{$x\log x$ condition} \label{xlogx}
\underset{i\in E}{\sum}\int_{1\leq \bra{\mbf{1}}{\mbf{z}}<\infty}\bra{\mbf{1}}{\mbf{z}}\ln(\bra{\mbf{1}}{\mbf{z}})\mu_i(\ud \mbf{z})<\infty.
\end{equation}
If $\mu_i(\ud \mbf{z})=\mbf{\pi}\Pi_i(\ud r)$, where $\mbf{z} = r\mbf{\pi}$ with $\Pi_i$ being a measure on $(0,\infty)$ and $\mbf{\pi}$ a fixed probability mass function on the type space $E$ in vector form, the following theorem  comes from Theorem 5.1 and Theorem 6.2 in \cite{rensongyang}. See also \cite[Theorem 6]{kp}.

\begin{theorem}\label{xlogx theorem}
Suppose that $\lambda_1>0$. The following assertions hold:
\begin{enumerate}
	\item If \eqref{xlogx} holds,
then for any $\p_{\mbf{x}}$,  $W_{\infty}$ is the $\mathbb{L}^1(\p_{\mbf{x}})$ limit of $W_t$ as $t\to\infty$.
	\item If \eqref{xlogx} doesn't hold,
then for any $\mbf{x}\in\R^d_{+}$, $W_\infty=0$, $\P_{\mbf{x}}$ a.s.
\end{enumerate}
\end{theorem}

We will prove this theorem in Section \ref{spine and martingale}. By using this theorem we obtain a strong law of large numbers for MCSBPs. This result matches perfectly classical analogues for multi-type Galton--Watson processes; see for example Theorem V.6.1 of \cite{an}.

\begin{theorem}\label{tstrong1}
Suppose that $\mbf{X}$ is a MCSBP with principal eigenvalue $\lambda_1>0$  and right and left  associated eigenvectors $\mbf{{\phi}},\mbf{\widehat{\phi}}\in\R_+^d$.
Then for any $\mbf{x}\in\R^d_+$,
\begin{equation*}\label{strong1}
\underset{t\rightarrow \infty}{\lim}  {\rm e}^ {-\lambda_1 t} \mbf{X}_t=W_{\infty}\mbf{\widehat{\phi}},\quad \P_{\mbf{x}}\mbox{ a.s.}
\end{equation*}
\end{theorem}

As a corollary we obtain the following  convergence on rates of types.

\begin{theorem}
For any $\mbf{x}\in\R^d_+\setminus\{\mbf{0}\}$ we have, conditional on non-extinction,
\begin{equation*}
\underset{t\rightarrow \infty}{\lim}\frac{\mbf{X}_t}{\bra{\mbf{1}}{\mbf{X}_t}}=\frac{\widehat{\mbf{\phi}}}{\bra{\mbf{1}}{\widehat{\mbf{\phi}}}},\qquad \P_{\mbf{x}}\mbox{ a.s.}
\end{equation*}
\end{theorem}

The remainder of this paper is structured as follows. In Section \ref{spine and martingale} we prove Theorem 1. The proof of Theorem 2 is given in Section \ref{section theorem 2}.

\section{Spine decomposition } \label{spine and martingale}
A now classical way to prove Theorem \ref{spine and martingale} is to find a spine decomposition for $\{\mbf{X}_t, t\geq 0\}$ under the
Doob $h$-transform
associated with $W$. More precisely, for any $\mbf{x}\in\R^d_+$, using the martingale $(W_t,t\geq 0)$ we define a new probability measure via
$$\left.\frac{\ud \widetilde{\P}_{\mbf{x}}}{\ud \p_{\mbf{x}}}\right|_{\mathcal{ F}_t}=\frac{1}{\bra{\mbf\phi}{ \mbf{x}}}W_t,\quad t\ge 0.$$
Where $\{\mathcal{ F}_t, t\ge 0\}$ is the natural filtration generated by  $\mbf{X}$.

Let $(\eta_t, t\geq 0)$  be a Markov chain  on $E$ with infinitesimal generator $\mbf{L}$, a $d\times d$ matrix defined  by
$${L}_{ij}=\frac{1}{\phi(i)}\left(B^T_{ij}-\Ind_{\{i=j\}}\lambda_1\right)\phi(j), \qquad i,j\in E.$$
Denote by $(\mbf{P}_i^{\mbf{\phi}}, i\in E)$ the probabilities of $\eta$ such that $\mbf{P}_i^{\mbf{\phi}}(\eta_0=i)=1$ for all $i\in \E$.

\begin{theorem}\label{spinesgpn}

\underline{}If $\mbf{X}$ is a MCSBP, then for any $\mbf{x}\in\R^d_+$ and $\mbf{f}\in \R_+^d$,
	\begin{align*}
	\widetilde{\e}_{\mbf{x}}&\left[{\rm e}^{-\langle \mbf{f}, \mbf{X}_t\rangle} \right]=\e_{\mbf{x}}\left[
	{\rm e}^{-\langle \mbf{f}, \mbf{X}_t\rangle}\right] \times\notag\\
	&\hspace{2cm}\mbf{E}_{\mbf{\phi x}}^{\mbf{\phi}}\left[\exp\left\{-\int_0^t \left(2c(\eta_s)v_{\eta_{s}}({t-s},\mbf{f})+	\int_{\R^d_+} {z}(\eta_s)(1- {\rm e}^ {-\langle\mbf{v}({t-s},\mbf{f}),\mbf{z}\rangle})\mu_{\eta_{s}}(\ud \mbf{z})\right)\ud s\right\}\right.\times
	\notag\\
	&\hspace{8cm}
	\underset{s\leq t}{\prod}{\mathcal{A}}_{\eta_{s-},\eta_s}^{t-s}\Bigg],
	\label{semigroupdecompn}\end{align*}
	where the matrices $\{{\mathcal{A}}^{ s}: s\geq 0\}$
are given by
	$${\mathcal{A}}^{ s}_{i,j}=\left(\Indi{i\neq j}\frac{1} {B^T_{ij}}
	\int_{\R^d_+}{z}(j) ( {\rm e}^ {-\langle\mbf{v}(s,\mbf{f}),\mbf{z}\rangle}-1){\mu_{i}(\ud \mbf{z})}+1\right),\quad i,j\in E,$$
	and
 $$\mbf{P}_{ \mbf{\phi x}}^{\mbf{\phi}}(\cdot) = \sum_{i \in E} \frac{ \phi(i)x(i)}{\langle {\mbf \phi}, {\mbf{x}}\rangle}\mbf{P}^{\mbf{\phi}}_i(\cdot)
 $$
	with associated expectation operator $\mbf{E}_{ \mbf{\phi x}}^{\mbf{\phi}}(\cdot)$.
\end{theorem}

\begin{proof}
	
	We start by noting that
	\begin{align*}
	\widetilde{\e}_{\mbf{x}}\left[{\rm e}^{-\langle \mbf{f}, \mbf{X}_t\rangle} \right]&=\frac{{\rm e}^{-\lambda_1 t}}{\langle  \mbf\phi,\mbf{x}\rangle}\e_{\mbf{x}}\left[\langle\mbf{\phi}, \mbf{X}_t\rangle
	{\rm e}^{-\langle \mbf{f}, \mbf{X}_t\rangle}\right].
	\end{align*}
	Replacing $\mbf{f}$ by $\mbf{f}+\lambda \mbf\phi$ in (\ref{laplace}) and (\ref{differential}) and differentiating with respect to $\lambda$ and then setting $\lambda  =0$, we obtain
	\begin{equation}
	\begin{split}
	\widetilde{\e}_{\mbf{x}}\left[{\rm e}^{-\langle \mbf{f}, \mbf{X}_t\rangle} \right]	&=	\e_{\mbf{x}}\left[{\rm e}^{-\langle \mbf{f}, \mbf{X}_t\rangle}\right]\underset{i\in E}{\sum}\frac{\phi(i)x(i)}{\langle \mbf\phi,\mbf{x}\rangle}\theta^t(i)\\
	&=
	\e_{\mbf{x}}\left[{\rm e}^{-\langle \mbf{f}, \mbf{X}_t\rangle}\right]\frac
	{\langle  \theta^t, {\mbf\phi}\circ{\mbf{x}}\rangle}{\langle \mbf\phi,\mbf{x}\rangle},
	\label{suma h}
	\end{split}
	\end{equation}
	where $\circ$ denotes element wise multiplication of vectors and, for $t\geq 0$, $
	\theta^t$ is the vector with entries
	\[
	\theta^t(i) := \frac{1}{
		\phi(i)}{\rm e}^{-\lambda_1 t}\left.\frac{\partial}{\partial \lambda} v_i(t,\mbf{f}+\lambda \mbf{\phi})\right|_{\lambda  =0}.  \qquad
	\]
By an integration by parts, using \eqref{differential} and \eqref{def psi} 
and that $(\mbf{B}\mbf{e}_i)_j=\mbf{B}_{ij}^T$, we get that $\theta^t(i)$ is also the
	unique  vector  solution to
	\begin{align*}
	\theta^t(i) &=1
	+\int_0^{t}	\frac{1}{\phi(i)}\big[( B^T-\lambda_1 I)(\mbf{\phi}\circ \theta^s)\big](i)\ud s
	-\int_0^{t}\theta^s(i)2c(i)v_i(s,\mbf{f})\ud s \nonumber\\
	&\hspace{2cm}+\int_0^{ t} \int_{\R^d_+}\bra{\theta^s}{\frac{\mbf{\phi}\circ \mbf{z}}{\phi(i)}}( {\rm e}^ {-\langle\mbf{v}(s,\mbf{f}),\mbf{z}\rangle}-1)\mu_i(\ud \mbf{z})\ud s.\nonumber
	\end{align*}

	\noindent Recall the definition of $\mbf{L}$ and note that it conforms to the definition of an intensity matrix of a Markov chain, thanks to the fact that $\mbf\phi$ is an eigenvector of $\mbf{B}^T$. A (vectorial) integration by parts in the spirit of Theorem 3.1.2 of \cite{dynkin},
	\begin{equation*}
	\begin{split}
	[{\theta}^t](i) =&[{\rm e}^{t\mbf{L}}\mbf{1}](i) -\int_0^{t}{\rm e}^{(t-s)\mbf{L}}\left[{\theta}^s(\cdot)\circ(2c(\cdot)v_\cdot(s,\mbf{f}))\right](i)\ud s \nonumber\\
	&\hspace{.5cm}+\int_0^{ t}
	{\rm e}^{(t-s)\mbf{L}}\left[
	\int_{\R^d_+}\bra{\theta^s}{\mbf{\phi}\circ \mbf{z}}( {\rm e}^ {-\langle\mbf{v}(s,\mbf{f}),\mbf{z}\rangle}-1)\frac{\mu_\cdot(\ud \mbf{z})}{\phi(\cdot)}\right](i)\ud s,\nonumber
	\end{split}
	\end{equation*}
where $\mbf{1} =(1,\cdots, 1)^T\in\mathbb{R}_+^d$.
	Then appealing to  the fact that $\{{\rm e}^{t\mbf{L}}:t\geq 0\}$ is the semigroup of $(\eta, \mbf{P}^{\mbf{\phi}}_i)$, $i\in E$,
	\begin{align}
	{\theta}^t(i) &=\mbf{E}_i^{\mbf{\phi}}[\mbf{1}] -\int_0^{t}\mbf{E}_i^{\mbf{\phi}}\left[2{\theta}^s(\eta_{t-s})c(\eta_{t-s})v_{\eta_{t-s}}(s,\mbf{f}))\right]\ud s \nonumber\\
	&\hspace{.5cm}+\int_0^{ t}
	\mbf{E}_i^{\mbf{\phi}}\left[
	\int_{\R^d_+}\bra{\theta^s}{\mbf{\phi}\circ \mbf{z}}( {\rm e}^ {-\langle\mbf{v}(s,\mbf{f}),\mbf{z}\rangle}-1)\frac{\mu_{\eta_{t-s}}(\ud \mbf{z})}{\phi(\eta_{t-s})}\right]\ud s.
\label{theprevious}
	\end{align}
	Next, we make a change of variable $u=t-s$ and separate the last inner product into  two parts. For all $s\leq t$ and $\mbf{z}\in \R_+^d$
	$$\bra{\theta^{t-s}}{\mbf{\phi}\circ \mbf{z}}=\theta^{t-s}(\eta_s)\phi(\eta_s) {z}(\eta_s)+\underset{j\neq \eta_s}{\sum}\theta^{t-s}(j)\phi(j) {z}(j). $$\
Therefore, \eqref{theprevious} is transformed into
	\begin{equation*}
	\begin{split}
	{\theta}^t(i) =	&\mbf{E}_i^{\mbf{\phi}}[g(\eta_t)] -\int_0^{t}\mbf{E}_i^{\mbf{\phi}}\left[2{\theta}^{t-s}(\eta_{s})c(\eta_{s})v_{\eta_{s}}({t-s},\mbf{f}))\right]\ud s \nonumber\\
	&\hspace{.5cm}+\int_0^{ t}
	\mbf{E}_i^{\mbf{\phi}}\left[ \theta^{t-s}(\eta_s)
	\int_{\R^d_+} {z}(\eta_s)( {\rm e}^ {-\langle\mbf{v}({t-s},\mbf{f}),\mbf{z}\rangle}-1)\mu_{\eta_{s}}(\ud \mbf{z})\right]\ud s\\
	&\hspace{.5cm}+\int_0^{ t}
	\mbf{E}_i^{\mbf{\phi}}\left[ \underset{j\neq \eta_s}{\sum}\theta^{t-s}(j)\frac{\phi(j)} {\phi(\eta_{s})}
	\int_{\R^d_+}{z}(j) ( {\rm e}^ {-\langle\mbf{v}({t-s},\mbf{f}),\mbf{z}\rangle}-1){\mu_{\eta_{s}}(\ud \mbf{z})}\right]\ud s.
	\end{split}
	\end{equation*}

	Recall that \eqref{inequality B} holds and hence, by applying
\cite[Lemma 6.1]{CRS}
 to the L\'evy system associated to $\mbf{L}$, using, in their notation, the functions
	$$q(s,i)=2c(i)v_{i}(s,\mbf{f})+	\int_{\R^d_+}{z}(i)(1- {\rm e}^ {-\langle\mbf{v}(s,\mbf{f}),\mbf{z}\rangle})\mu_{i}(\ud \mbf{z})$$
	and
	$$F(s,i,j)=\ln\left(	\Indi{i\neq j}\frac{1}{B_{ji}}\int_{\R^d_+}{z}(j) ( {\rm e}^ {-\langle\mbf{v}(s,\mbf{f}),\mbf{z}\rangle}-1)\mu_{i}(\ud \mbf{z})+1\right),$$
	we obtain
	
	{\small \begin{align*}
		\theta^{t}(i)= &\mbf{E}_{i}^{\mbf{\phi}}\left[\exp\left\{-\int_0^t \left(2c(\eta_s)v_{\eta_{s}}({t-s},\mbf{f})+	\int_{\R^d_+} {z}(\eta_s)(1- {\rm e}^ {-\langle\mbf{v}({t-s},\mbf{f}),\mbf{z}\rangle})\mu_{\eta_{s}}(\ud \mbf{z})\right)\ud s\right\} \underset{s\leq t}{\prod}{\mathcal{A}}^{t-s}_{\eta_{s-},\eta_s}\right].
		\end{align*}}
\end{proof}

Theorem \ref{spinesgpn} suggests that  the process $(\mbf{X}_t,\widetilde{\p}_{\mbf{x}})$ is equal in law to a process $\{\mbf{\Gamma}_t : t\geq 0\}$,  whose law is henceforth denoted by ${\rm P}_{\mbf{x}}$, $\mbf{x}\in\mathbb{R}^d_+$, where
\begin{equation*}
\mbf{\Gamma}_t = \mbf{X}'_t + \sum_{s\leq t:{c}} \mbf{X}^{{\rm c}, s}_{t-s} + \sum_{s\leq t: {\rm d}}\mbf{X}^{{\rm d}, s}_{t-s} + \sum_{s\leq t: {\rm j}}\mbf{X}^{{\rm j}, s}_{t-s}, \qquad t\geq 0,
\label{spine}
\end{equation*}
such that  $\mbf{X}'$ is an independent copy of $(\mbf{X}_t,\p_{\mbf{x}})$ and the processes $\mbf{X}^{{\rm c}, s}_{\cdot}$, $\mbf{X}^{{\rm d}, s}_{\cdot}$ and $\mbf{X}^{{\rm j}, s}_{\cdot}$ are defined through a process of immigration as follows: Given the path of the Markov chain $(\eta, \mbf{P}^{\mbf{\phi}}_{\mbf{\phi}\mbf{x}})$,

\bigskip

\noindent {\bf [continuous immigration]}  in a Poissonian way  a $\psi$-MCSBP $\mbf{X}^{{\rm c},s}_{\cdot}$ immigrates at $(s, \eta_s)$ with rate $\ud s\times 2c(\eta_s)\ud\mathbf{N}_{\eta_s}$,

\bigskip

\noindent {\bf [discontinuous immigration]} in a Poissonian way  a $\psi$-MCSBP $\mbf{X}^{{\rm d}, s}_\cdot$ immigrates at $(s, \eta_s)$ with rate $\ud s\times \int_{\R^d_+} {z}(\eta_s)\mu_{\eta_{s}}(\ud \mbf{z})\p_{\mbf{z}}$
\bigskip

\noindent {\bf [jump immigration]} at each jump time $s$ of $\eta$, a $\psi$-MCSBP $\mbf{X}^{{\rm j}, s}_\cdot$ immigrates  at $(s, \eta_s)$
with law $ \int_{\R^d_+} \nu_{\eta_{s-},\eta_s}(\ud \mbf{z})\p_{\mbf{z}}$, where for each $i,j\in E$,
	$$\nu_{i,j}(\ud \mbf{z})=\Indi{i\neq j}\frac{1} {B^T_{ij}}
	{z}(j) \mu_{i}(\ud \mbf{z})+\left(1-\Indi{i\neq j}\frac{1} {B^T_{ij}}
	\int_{\R^d_+}{v}(j) \mu_{i}(\ud \mbf{v})\right)\delta_{\mbf{0}}(\ud \mbf{z}).$$
\bigskip
\noindent Given $\eta$, the above three immigration  processes are independent.

In the above description,
the quantity $\mathbf{N}_i$ is the excursion measure of the $\psi$-MCSBP corresponding to
$\p_{\mbf{e}_i}$.
To be more precise, \cite{DK04} showed that associated to the laws
$\{\p_{\mbf{e}_i}: i\in E\}$ are the measures $\{\mathbf{N}_i:i\in E\}$,
defined on the same measurable space, which satisfy
$$\mathbf{N}_i(1-{\rm e}^{-\langle \mbf{f}, \mbf{X}_t \rangle})=-\log\e_{\mbf{e}_i}({\rm e}^{-\langle \mbf{f}, \mbf{X}_t \rangle}),$$
for all $t\geq 0$.  A particular feature of $\mathbf{N}_i$ that we shall use later is that
\begin{equation}
\mathbf{N}_i(\langle \mbf{f}, \mbf{X}_t \rangle) = \e_{\mbf{e}_i}(\langle \mbf{f}, \mbf{X}_t \rangle).
\label{N}
\end{equation}

Observe that the processes $\mbf{X}^{\mbf{c}}$, $\mbf{X}^{\mbf{d}}$ and $\mbf{X}^{\mbf{j}}$ are initially zero valued, therefore,
if $\mbf{\Gamma}_0=\mbf{x}$ then $\mbf{X}_0'=\mbf{x}$.
The following result corresponds to a classical spine decomposition, albeit now for the setting of an MCSBP. Note, we henceforth refer to the process $\eta$ as the {\it spine}. By following the same proof as Theorem 5 in  \cite{kp} we can easily establish the next result.

\begin{theorem}[Spine decomposition]\label{gamma}
For any $\mbf{x}\in \R^d_+$,
$(\mbf{\Gamma}_t, t\ge 0; {\rm P}_{\mbf{x}})$ is equal in law to $(\mbf{X}_t,t\ge 0; \widetilde{\p}_{\mbf{x}})$.
\end{theorem}

\noindent For the sake of brevity, we leave the proof to the reader.

\subsection{Proof of Theorem \ref{xlogx theorem}}\label{spine2}

	We follow a well established line of reasoning. We know that $1/W_t$ is a positive $\widetilde{\p}_{\mbf{x}}$-supermartingale and hence $\lim_{t\to\infty}W_t$ exists $\widetilde{\p}_{\mbf{x}}$-almost surely. Therefore $W_t$ converges in $\mathbb{L}^1(\P_{\mbf{x}})$ to a non-degenerated limit as soon as we prove that
$\widetilde{\p}_{\mbf{x}}(\liminf_{t\to\infty}W_t<\infty)=1$.
	
	We consider the spine decomposition in Theorem \ref{gamma}. Given the spine $\eta$, let us write  $(s,\mbf{I}_s^{\rm d},\mbf{I}_s^{\rm j})_{s\geq 0}$ for the process of immigrated vector along the spine (i.e. $\mbf{I}_s^{\rm d}=\mbf{X}_0^{{\rm d},s}$ and $\mbf{I}_s^{\rm j}=\mbf{X}_0^{{\rm j},s}$). Then $(s,\mbf{I}_s^{\rm d})$ is Poissonian with intensity $\ud s\times {z}(\eta_s)\mu_{\eta_{s}}(\ud \mbf{z})$ and, if $s$ is such that $\eta_{s-}\neq \eta_s$ then $\mbf{I}_s^{\rm j}$ is distributed according to $ \nu_{\eta_{s-},\eta_s}.$
	Let  $\mathcal{S} = \sigma(\eta_s, (s,\mbf{I}_s^{\rm d},\mbf{I}_s^{\rm j}), s\geq 0)$  be the sigma algebra which informs the location of the spine	and the vector issued at each immigration time and write
\[
	Z_t: = {\rm e}^{-\lambda_1 t}\langle \mbf{\phi}, \mbf{\Gamma}_t\rangle.
	\]

Since $(\mbf{\Gamma}_\cdot, {\rm P}_{\mbf{x}})$ is equal in law to $(\mbf{X}_{\cdot}, \widetilde{\p}_{\mbf{x}})$, to prove $\widetilde{\p}_{\mbf{x}}(\liminf_{t\to\infty}W_t<\infty)=1$, we only need to prove that
\begin{equation}\label{limit-Z}{\rm P}_{\mbf{x}}(\liminf_{t\to\infty}Z_t<\infty)=1.\end{equation}
	By Fatou's Lemma
	\[
	{\rm E}_{\mbf{x}}
  [\liminf_{t\to\infty}Z_t | \mathcal{S}]\leq \liminf_{t\to\infty}{\rm E}_{\mbf{x}}[Z_t|\mathcal{S}].
	\]
	It therefore remains to show that $\liminf_{t\to\infty}{\rm E}_{\mbf{x}}[Z_t| \mathcal{S}]<\infty$,  $\widetilde{\p}_{\mbf{x}}$ a.s.

By using that the processes $\mbf{X}^{{\rm c}, s}_{\cdot}$, $\mbf{X}^{{\rm d}, s}_{\cdot}$ conditioned on $\mathcal{S}$ are Poissonian, the description of $\mbf{X}^{{\rm j}, s}_{\cdot}, \mbf{I}^{{\rm j}}_{\cdot}$ and $\mbf{I}^{{\rm d}}_{\cdot}$,  formula  \eqref{N} and that $W$ is a martingale, we have that for  $t\geq 0$ and $\mbf{x}\in\R^d_+$,
	\begin{align*}
	\liminf_{t\to\infty}{\rm E}_{\mbf{x}}[Z_t|\mathcal{S}]=& \langle\mbf{\phi}, \mbf{x}\rangle +\int_0^\infty 2 c(\eta_s) {\rm e}^{-\lambda_1 s}\phi(\eta_s)
	{\ud}s
	+	\underset{0\leq s}{\sum} {\rm e}^ {-\lambda_1 s}\bra{\mbf{\phi}}{\mbf{I}_s^{\rm d}}	
	+	\underset{0\leq s}{\sum} {\rm e}^ {-\lambda_1 s}\bra{\mbf{\phi}}{\mbf{I}_s^{\rm j}}.
	\end{align*}
	Since $\lambda_1>0$, the first integral is finite.
We need to prove that the other sums are finite a.s.
	In order to do it, we will decompose the sum in small jumps: $\{(s,\mbf{I}_s^{\rm d}):\bra{\mbf{1}}{\mbf{I}_s^{\rm d}}\leq  {\rm e}^ {\lambda_1 s}\}$ and big jumps: $\{(s,\mbf{I}_s^{\rm d}):\bra{\mbf{1}}{\mbf{I}_s^{\rm d}}>  {\rm e}^ {\lambda_1 s}\}$,
and we handle $\mbf{I}_s^{\rm j}$ in similar way.
Denote
$M_{\mbf{\phi}}:=\max\{\phi(i):i\in E\}$. For the small jumps, by the definition of $\nu$,
	\begin{align*}
&{\rm E}_{\mbf{x}}
\left[\underset{0\leq s}{\sum} {\rm e}^ {-\lambda_1 s}\bra{\mbf{\phi}}{\mbf{I}_s^{\rm d}}	\Indi{\bra{\mbf{1}}{\mbf{I}_s^{\rm d}}\leq  {\rm e}^ {\lambda_1 s}}
		+	\underset{0\leq s}{\sum} {\rm e}^ {-\lambda_1 s}\bra{\mbf{\phi}}{\mbf{I}_s^{\rm j}}\Indi{\bra{\mbf{1}}{\mbf{I}_s^{\rm j}}\leq  {\rm e}^ {\lambda_1 s}}\right]\\
=&	
{\rm E}_{\mbf{x}}
\left[	\int_0^\infty\hspace{-.2cm}\int_{\bra{\mbf{1}}{\mbf{z}}\leq  {\rm e}^ {\lambda_1 s}}\hspace{-.2cm}  {\rm e}^ {-\lambda_1 s}{z}(\eta_s)\bra{\mbf{\phi}}{\mbf{z}}\mu_{\eta_{s}}(\ud \mbf{z})\ud s
	+\int_0^\infty\hspace{-.2cm}\int_{\bra{\mbf{1}}{\mbf{z}}\leq  {\rm e}^ {\lambda_1 s}}\hspace{-.2cm}  {\rm e}^ {-\lambda_1 s}\bra{\mbf{\phi}}{\mbf{z}}\nu_{\eta_{s-},\eta_s}(\ud \mbf{z})\ud s\right]\\
\leq &
CM_{\mbf{\phi}}\underset{i\in E}{\sum}\int_0^\infty\int_{\bra{\mbf{1}}{\mbf{z}}\leq  {\rm e}^ {\lambda_1 s}} {\rm e}^ {-\lambda_1 s}\bra{\mbf{1}}{\mbf{z}}^2\mu_{i}(\ud \mbf{z})\ud s,
	\end{align*}
	where $C$ is a positive constant.
For each $i\in E$,
	\begin{equation*}
	\begin{split}
	\int_0^\infty\int_{\bra{\mbf{1}}{\mbf{z}}\leq  {\rm e}^ {\lambda_1 s}} {\rm e}^ {-\lambda_1 s}\bra{\mbf{1}}{\mbf{z}}^2\mu_{i}(\ud \mbf{z})\ud s\leq
	\int_{\bra{\mbf{1}}{\mbf{z}}\leq 1}\bra{\mbf{1}}{\mbf{z}}^2\mu_i(\ud \mbf{z})+\int_{\bra{\mbf{1}}{\mbf{z}}> 1}\bra{\mbf{1}}{\mbf{z}}\mu_i(\ud \mbf{z})
	<\infty.
	\end{split}
	\end{equation*}
	Therefore,
$$	\underset{0\leq s}{\sum} {\rm e}^ {-\lambda_1 s}\bra{\mbf{\phi}}{\mbf{I}_s^{\rm d}}	\Indi{\bra{\mbf{1}}{\mbf{I}_s^{\rm d}}\leq  {\rm e}^ {\lambda_1 s}}
	+	\underset{0\leq s}{\sum} {\rm e}^ {-\lambda_1 s}\bra{\mbf{\phi}}{\mbf{I}_s^{\rm j}}\Indi{\bra{\mbf{1}}{\mbf{I}_s^{\rm j}}\leq  {\rm e}^ {\lambda_1 s}}<\infty, \qquad
{\rm P}_{\mbf{x}}\mbox{ a.s.}
$$

	For the big jumps, using Fubini's Theorem, we get
	\begin{equation*}
	\begin{split}
	&{\rm E}_{\mbf{x}}
\left[\underset{0\leq s}{\sum}	\Indi{\bra{\mbf{1}}{\mbf{I}_s^{\rm d}}>  {\rm e}^ {\lambda_1 s}}
	+	\underset{0\leq s}{\sum}\Indi{\bra{\mbf{1}}{\mbf{I}_s^{\rm j}}>  {\rm e}^ {\lambda_1 s}}\right]
\\
=&
    {\rm E}_{\mbf{x}}
\left[	\int_0^\infty\hspace{-.2cm}\int_{\bra{\mbf{1}}{\mbf{z}}>  {\rm e}^ {\lambda_1 s}}\hspace{-.2cm} {z}(\eta_s)\mu_{\eta_{s}}(\ud \mbf{z})\ud s
	+\int_0^\infty\hspace{-.2cm}\int_{\bra{\mbf{1}}{\mbf{z}}>  {\rm e}^ {\lambda_1 s}}\hspace{-.2cm} \nu_{\eta_{s-},\eta_s}(\ud \mbf{z})\ud s\right]\\
\leq &
C \underset{i\in E}{\sum}\int_0^\infty\int_{\bra{\mbf{1}}{\mbf{z}}>  {\rm e}^ {\lambda_1 s}}\bra{\mbf{1}}{\mbf{z}}\mu_{i}(\ud \mbf{z})\ud s\\
=&
\frac{C}{\lambda_1} \underset{i\in E}{\sum}\int_{\bra{\mbf{1}}{\mbf{z}}> 1}\bra{\mbf{1}}{\mbf{z}}\ln(\bra{\mbf{1}}{\mbf{z}})\mu_i(\ud \mbf{z})<\infty.
	\end{split}
	\end{equation*}
This implies that ${\rm P}_{\mbf{x}}$ a.s. we have finitely many big jumps and therefore
	$$	\underset{0\leq s}{\sum} {\rm e}^ {-\lambda_1 s}\bra{\mbf{\phi}}{\mbf{I}_s^{\rm d}}	\Indi{\bra{\mbf{1}}{\mbf{I}_s^{\rm d}}>  {\rm e}^ {\lambda_1 s}}
	+	\underset{0\leq s}{\sum} {\rm e}^ {-\lambda_1 s}\bra{\mbf{\phi}}{\mbf{I}_s^{\rm j}}\Indi{\bra{\mbf{1}}{\mbf{I}_s^{\rm j}}>  {\rm e}^ {\lambda_1 s}}<\infty, \qquad
{\rm P}_{\mbf{x}}\mbox{ a.s.}
	$$
	So, $\liminf_{t\to\infty}{\rm E}_{\mbf{x}}[Z_t| \mathcal{S}]<\infty$, ${\rm P}_{\mbf{x}}$ a.s.

	 Now, we will prove the second part of the Theorem. Since
	$$\widetilde{\p}_{\mbf{x}}(\underset{t\rightarrow\infty}{\limsup}\ W_t=\infty)={\rm P}_{\mbf{x}}(\underset{t\rightarrow\infty}{\limsup}\ Z_t=\infty),$$
     if we prove that
	\begin{equation}\label{limsup}
	\underset{t\rightarrow\infty}{\limsup}\ Z_t=\infty \qquad  {\rm P}_{\mbf{x}}\mbox{ a.s.}
	\end{equation}
	then $\widetilde{\p}_{\mbf{x}}$ and ${\p}_{\mbf{x}}$ are singular and hence
	$$\p_{\mbf{x}}(\underset{t\rightarrow\infty}{\limsup}\ W_t=0)=1.$$
	It remains to show \eqref{limsup}. Suppose that for a fixed $i\in E$,
	$$\int_{\bra{\mbf{1}}{\mbf{z}}>1}\bra{\mbf{1}}{\mbf{z}}\ln(\bra{\mbf{1}}{\mbf{z}})\mu_{i}(\ud \mbf{z})=\infty.$$
	We will divide the proof in two parts.
	
	\bigskip

	(i)		First assume that $$\int_{\bra{\mbf{1}}{\mbf{z}}>1}z(i)\ln(\bra{\mbf{1}}{\mbf{z}})\mu_{i}(\ud \mbf{z})=\infty.$$
		
		Denote by $\mathcal{T}$ the set of times at which we immigrate $(s,\mbf{X}^{d,s}_{\cdot})$  along the spine, then for $s\in \mathcal{T}$,
		$$Z_s\geq  {\rm e}^ {-\lambda_1 s} \bra{\mbf{\phi}}{\mbf{X}^{d,s}_0}.$$

To prove \eqref{limsup}, we only need to prove that
$$
\underset{\mathcal{T}\ni s\rightarrow\infty}{\limsup}\  {\rm e}^ {-\lambda_1 s} \bra{\mbf{\phi}}{\mbf{X}^{d,s}_0}=\infty.
$$
Since $ {\rm e}^ {-\lambda_1 s} \bra{\mbf{\phi}}{\mbf{X}^{d,s}_0}\ge m_{\mbf{\phi}} {\rm e}^ {-\lambda_1 s} \bra{\mbf{1}}{\mbf{X}^{d,s}_0}$, where $m_{\mbf{\phi}}:=\min\{\phi(i): i\in E\}$, we only need to prove that
\begin{equation}\label{limsup-infty}
\underset{\mathcal{T}\ni s\rightarrow\infty}{\limsup}\  {\rm e}^ {-\lambda_1 s} \bra{\mbf{1}}{\mbf{X}^{d,s}_0}=\infty.
\end{equation}

For $ T,K>1$, define the subsets
		$$A_{T,K}:=\sharp \{s\in\mathcal{T}\cap (T,\infty): \bra{\mbf{1}}{\mbf{X}^{d,s}_0} >K {\rm e}^ {\lambda_1 s}\}.$$
Then,  given the Markov chain $(\eta, \mbf{P}^{\mbf{\phi}}_{\mbf{\phi}\mbf{x}})$, $A_{T,K}$ is a Poisson random variable with parameter
$\int_T^\infty \int_{\bra{\mbf{1}}{\mbf{z}} >K {\rm e}^ {\lambda_1 s}} {z}(\eta_s)\mu_{\eta_{s}}(\ud \mbf{z}){\rm d}s$. To prove \eqref{limsup-infty}, we only need to prove that
$$\int_T^\infty {\rm d}s\int_{\bra{\mbf{1}}{\mbf{z}} >K {\rm e}^ {\lambda_1 s}} {z}(\eta_s)\mu_{\eta_{s}}(\ud \mbf{z})=\infty,\qquad  \mbf{P}^{\mbf{\phi}}_{\mbf{\phi}\mbf{x}}\text{ a.s.}$$ for all $T,K>1$.
		Observe that $\eta$ is ergodic. Then  there exists $C_i>0$ and random $S$ such that
		$$\int_0^t \Indi{\eta_s=i}\ud s\geq C_it \qquad\mbox{ for all } t\geq S\qquad  \mbf{P}^{\mbf{\phi}}_{\mbf{\phi}\mbf{x}}\text{ a.s.}$$
		Let us denote by $R=\max\{K {\rm e}^ {\lambda_1 T}, K {\rm e}^ {\lambda_1 S}\}$. Then, by Fubini's Theorem and the previous inequality,
		\begin{equation*}
		\begin{split}
		\int_T^\infty \int_{\bra{\mbf{1}}{\mbf{z}} >K {\rm e}^ {\lambda_1 s}} {z}(\eta_s)\mu_{\eta_{s}}(\ud \mbf{z})\ud s
        &\geq\int_T^\infty\int_{\bra{\mbf{1}}{\mbf{z}} >K {\rm e}^ {\lambda_1 s}} {z}(\eta_s)\Indi{\eta_{s}=i}\mu_{\eta_{s}}(\ud \mbf{z})\ud s\\
		&\geq\frac{C_i}{\lambda_1}\int_{\bra{\mbf{1}}{\mbf{z}} >R}{z}(i)\ln(\bra{\mbf{1}}{\mbf{z}})\mu_{i}(\ud \mbf{z})-D=\infty,
		\end{split}
		\end{equation*}
		where
		 $		 D=(\lambda_1^{-1}\ln(K)+T)
		 \int_{\bra{\mbf{1}}{\mbf{z}} >R}{z}(i)\mu_{i}(\ud \mbf{z})<\infty$.
Therefore we have \eqref{limsup}.
\bigskip

(ii)		Next suppose that, for  $j\neq i$,
		$$\int_{\bra{\mbf{1}}{\mbf{z}}>1}z(j)\ln(\bra{\mbf{1}}{\mbf{z}})\mu_{i}(\ud \mbf{z})=\infty.$$
		By inequality \eqref{inequality B} and the ergodicity of $\eta$, the set
		$
		\tau:=
		\{s\geq 0: \eta_{s-}=i, \eta_{s}=j\}$ is not bounded. For all $s\in \mathcal{T}$
		$$Z_s\geq  {\rm e}^ {-\lambda_1 s} \bra{\mbf{\phi}}{\mbf{X}^{j,s}_0}$$
		Let us denote by $\tau_1,\tau_2,\cdots$ a enumeration of this times in increasing order.
By applying the distribution of $\mbf{X}^{j,\cdot}$ in the jump times $\tau_n$ we have that for all $K>0$,
			\begin{equation*}
		\begin{split}
		\underset{n\geq 1}{\sum}{\rm P}_{\mbf{x}}( {\rm e}^ {-\lambda_1 \tau_n}\bra{\mbf{1}}{\mbf{X}^{j,\tau_n}_0}\geq K)&= 	\underset{n\geq 1}{\sum}{\rm E}_{\mbf{x}}[{\rm P}_{\mbf{x}}( {\rm e}^ {-\lambda_1 \tau_n}\bra{\mbf{1}}{\mbf{X}^{j,\tau_n}_0}\geq K\mid \tau_n)]\\
		&=	\frac{1}{B_{ij}^T}\underset{n\geq 1}{\sum} {\rm E}_{\mbf{x}}\left[\int_{\R_+^d}\Indi{ {\rm e}^ {-\lambda_1 \tau_n}\bra{\mbf{1}}{\mbf{z}}\geq K}z(j)\mu_i(\ud \mbf{z})\right]\\
		&=	\frac{1}{B_{ij}^T}{\rm E}_{\mbf{x}}\left[\underset{n\geq 1}{\sum} \int_{\R_+^d}\Indi{ {\rm e}^ {-\lambda_1 \tau_n}\bra{\mbf{1}}{\mbf{z}}\geq K}z(j)\mu_i(\ud \mbf{z})\right]
		\end{split}
\end{equation*}
By renewal theory, there exist $A\in(0,\infty)$  and a subset $\Omega_1$ with ${\rm P}_{\mbf{x}}(\Omega_1)=1$ such that for all $\omega\in \Omega_1$ there exists $N=N(\omega)>0$ such that
		$$\tau_n\leq A n, \qquad \mbox{ for all } \ n\geq N.$$
	Then,
		$$\underset{n\geq 1}{\sum}\int_{\R_+^d}\Indi{ {\rm e}^ {-\lambda_1 \tau_n}\bra{\mbf{1}}{\mbf{z}}\geq K}z(j)\mu_i(\ud \mbf{z})\\
		\geq \underset{n\geq N}{\sum}\int_{\R_+^d}\Indi{ {\rm e}^ {-\lambda_1 An}\bra{\mbf{1}}{\mbf{z}}\geq K}z(j)\mu_i(\ud \mbf{z}).$$
By the integral test criterion for series, the previous series is divergent since
			\begin{equation*}
		\begin{split}
\int_{N}^{\infty}\int_{\R_+^d}\Indi{ {\rm e}^ {-\lambda_1 As}\bra{\mbf{1}}{\mbf{z}}\geq K}z(j)\mu_i(\ud \mbf{z})\ud s
		&=\frac{1}{A\lambda_1}\int_{\bra{\mbf{1}}{\mbf{z}} >K}{z}(j)\ln(\bra{\mbf{1}}{\mbf{z}})\mu_{i}(\ud \mbf{z})-D=\infty,
		\end{split}
		\end{equation*}
		where
$D=((\lambda_1A)^{-1}\ln(K)+N)\int_{\bra{\mbf{1}}{\mbf{z}} >K}\mbf{z}(i)\mu_{i}(\ud \mbf{z})<\infty$. This implies
		$$	\underset{n\geq 1}{\sum}{\rm P}_{\mbf{x}}( {\rm e}^ {-\lambda_1 \tau_n}\bra{\mbf{1}}{\mbf{X}^{j,\tau_n}_0}\geq K)={\rm E}_{\mbf{x}}\left[\underset{n\geq 1}{\sum} \int_{\R_+^d}\Indi{ {\rm e}^ {-\lambda_1 \tau_n}\bra{\mbf{1}}{\mbf{z}}\geq K}z(j)\mu_i(\ud \mbf{z})\right]=\infty.$$
	Therefore, by the Borel Cantelli Lemma,
$$\underset{n\rightarrow \infty}{\limsup}\  {\rm e}^ {-\lambda _1 \tau_n}\bra{\mbf{1}}{\mbf{X}^{j,\tau_n}_0}\geq K,\qquad {\rm P}_{\mbf{x}}\mbox{ a.s.}$$
for any $K>0$, and then
	$$\underset{n\rightarrow \infty}{\limsup}\  {\rm e}^ {-\lambda _1 \tau_n}\bra{\mbf{1}}{\mbf{X}^{j,\tau_n}_0}=\infty.$$
	Hence we have
	$$\underset{s\rightarrow\infty}{\limsup}\ Z_s\geq\underset{n\rightarrow \infty}{\limsup}\  {\rm e}^ {-\lambda _1 \tau_n}\bra{\mbf{\phi}}{\mbf{X}^{j,\tau_n}_0}
	=\infty,$$
	which  says that \eqref{limsup} holds.\hfill$\square$

\section{Proof of Theorem \ref{tstrong1}}\label{section theorem 2}
Suppose that \eqref{xlogx} doesn't hold, then
$$ \underset{t\rightarrow \infty}{\limsup}\  {\rm e}^ {-\lambda_1 t} \mbf{X}_t(i)\leq \underset{t\rightarrow \infty}{\lim}\frac{1}{\phi(i)} {\rm e}^ {-\lambda_1 t} \bra{\mbf{\phi}}{\mbf{X}_t}=0 \qquad
{\P}_{\mbf{x}}\mbox{ a.s.}
$$
We therefore focus on the case when \eqref{xlogx} holds. In order to do this, we will separate small jumps from big jumps in the Poisson measures. More precisely, for each $i\in E$, let us define the Poisson random measures
$$N_i^{(1)}(\ud s,\ud \mbf{z}, \ud r):=\Indi{\bra{\mbf{1}}{\mbf{z}}\leq  {\rm e}^ {\lambda_1 s}}N_i(\ud s,\ud \mbf{z}, \ud r)$$
and
$$N_i^{(2)}(\ud s,\ud \mbf{z}, \ud r):=\Indi{\bra{\mbf{1}}{\mbf{z}}>  {\rm e}^ {\lambda_1 s}}N_i(\ud s,\ud \mbf{z}, \ud r)$$
and denote by $\widetilde{N}_i^{(1)}$ and $\widetilde{N}_i^{(2)}$ their compensated versions, respectively.

\noindent We are going to compute the proof of Theorem \ref{tstrong1} in three steps. First, in lattice times, we will approximate the value of the limit by the value of the limit of a conditional expectation. With this relation, we are going to find our limit in lattice times. And finally in the third step, we will extend the result to continuous times.

\subsection{Proof for lattice times}
First, we will prove Theorem \ref{tstrong1} in lattice times.
For each $\delta>0$, consider the lattice times $n\delta$, $n\in \N$.
 We will approximate the value of the limit by the value of the limit of a conditional expectation.

\begin{lemma}
	\label{lemma conditional} If \eqref{xlogx} holds, then for any $m\in \N$, $\sigma>0$ and
$\mbf{x}\in \R^d_+,$
	\begin{equation*}
	\lim_{n\rightarrow \infty} {\rm e}^ {-\lambda_1(n+m)\sigma}\mbf{X}_{(n+m)\sigma}-\Expconx{\mbf{x}}{ {\rm e}^ {-\lambda_1(n+m)\sigma}\mbf{X}_{(n+m)\sigma}}{\mathcal{F}_{n\sigma}}=0, \quad \mbox{ in } \mathbb{L}^1(\p_{\mbf{x}}) \mbox{ and } \p_{\mbf{x}}\mbox{ a.s.}
	\end{equation*}
\end{lemma}

\begin{proof}
The result is true if we prove that for all $k\in E$,
\begin{equation*}
	\lim_{n\rightarrow \infty} {\rm e}^ {-\lambda_1(n+m)\sigma}\bra{\mbf{e}_k}{\mbf{X}_{(n+m)\sigma}}-\Expconx{\mbf{x}}{ {\rm e}^ {-\lambda_1(n+m)\sigma}\bra{\mbf{e}_k}{\mbf{X}_{(n+m)\sigma}}}{\mathcal{F}_{n\sigma}}=0, \quad \mbox{ in } \mathbb{L}^1(\p_{\mbf{x}}) \mbox{ and } \p_{\mbf{x}}\mbox{ a.s.}
\end{equation*}
Let $s,t\geq 0$. By the Markov property we have
$$ {\rm e}^ {-\lambda_1(t+s)}\bra{\mbf{e}_k}{\mbf{X}_{t+s}}-\Expconx{\mbf{x}}{ {\rm e}^ {-\lambda_1(t+s)}\bra{\mbf{e}_k}{\mbf{X}_{t+s}}}{\mathcal{F}_{t}}= {\rm e}^ {-\lambda_1(t+s)}\bra{\mbf{e}_k}{\mbf{X}_{t+s}}- {\rm e}^ {-\lambda_1(t+s)}\bra{\mbf{M}(s)\mbf{e}_k}{\mbf{X}_{t}}.$$
Now, applying equation \eqref{useful equation} to the times $t$ and $t+s$, we obtain
\begin{align}
&
{\rm e}^ {-\lambda_1(t+s)}\bra{\mbf{e}_k}{\mbf{X}_{t+s}}- {\rm e}^ {-\lambda_1(t+s)}\bra{\mbf{M}(s)\mbf{e}_k}{\mbf{X}_{t}}\notag\\
=&
\underset{i\in E}{\sum}  {\rm e}^ {-\lambda_1(t+s)}\int_t^{t+s}[\mbf{M}(t+s-u)\mbf{e}_k]_i\sqrt{2c_iX_{u,i}}\ud W_{u,i}\notag\\
&+\underset{i\in E}{\sum} {\rm e}^ {-\lambda_1(t+s)}\int_t^{t+s}\int_{\R_+^d}\int_0^{\infty} \bra{\mbf{M}(t+s-u)\mbf{e}_k}{\mbf{z}}\Indi{l\leq X_{u-,i}}\widetilde{N}_i(\ud u,\ud \mbf{z}, \ud l)\notag\\
=&
C_{t,t+s}(\mbf{e}_k)+S_{t,t+s}(\mbf{e}_k)+B_{t,t+s}(\mbf{e}_k),
\label{1-3}
\end{align}
where
\begin{equation*}\begin{split}
C_{t,t+s}(\mbf{e}_k)&:=\underset{i\in E}{\sum}  {\rm e}^ {-\lambda_1(t+s)}\int_t^{t+s}[\mbf{ M}(t+s-u)\mbf{e}_k]_i\sqrt{2c_iX_{u,i}}\ud W_{u,i},\label{Ct,t+s} \\
S_{t,t+s}(\mbf{e}_k)&:=\underset{i\in E}{\sum} {\rm e}^ {-\lambda_1(t+s)}\int_t^{t+s}\int_{\R_+^d}\int_0^{\infty} \bra{\mbf{M}(t+s-u)\mbf{e}_k}{\mbf{z}}\Indi{r\leq X_{u-,i}}\widetilde{N}_i^{(1)}(\ud u,\ud \mbf{z}, \ud r),\\
B_{t,t+s}(\mbf{e}_k)&:=\underset{i\in E}{\sum} {\rm e}^ {-\lambda_1(t+s)}\int_t^{t+s}\int_{\R_+^d}\int_0^{\infty} \bra{\mbf{M}(t+s-u)\mbf{e}_k}{\mbf{z}}\Indi{r\leq X_{u-,i}}\widetilde{N}^{(2)}_i(\ud u,\ud \mbf{z}, \ud r).
\end{split}
\end{equation*}
To complete the proof, we need to control the convergence of the above three terms.

\bigskip

\noindent\underline{\it (i) Lattice convergence of $C_{t,t+s}(\mbf{e}_k)$}:
We will show that  for any $k\in E$, $m\in \N$, $\sigma>0$ and $\mbf{x}\in\R^d_+$
\begin{equation*}
\lim_{n\rightarrow \infty}C_{ n\sigma,(n+m)\sigma}(\mbf{e}_k)=0, \qquad \mbox{ in } \mathbb{L}^2(\p_{\mbf{x}}) \mbox { and } \p_{\mbf{x}}\mbox{ a.s.}
\end{equation*}

	First note that  for $t\in [n\sigma,(n+m)\sigma]$, the process
	\begin{equation}\label{c n,m}
	C_t^{(n,m,\sigma)}:=\underset{i\in E}{\sum}  {\rm e}^ {-\lambda_1(n+m)\sigma}\int_{n\sigma}^{t}[\mbf{M}((n+m)\sigma-u)\mbf{e}_k]_i\sqrt{2c_iX_{u,i}}\ud W_{u,i}
		\end{equation}
	 is a continuous local martingale with quadratic variation given by
	$$ \underset{i\in E}{\sum}  {\rm e}^ {-2\lambda_1(n+m)\sigma}\int_{n\sigma}^{t}([\mbf{M}((n+m)\sigma-u)\mbf{e}_k]_i)^2{2c_iX_{u,i}}\ud u.$$
	Then, by taking $t=(n+m)\sigma$, we have
\begin{equation*}\begin{split}
\Expx{\mbf{x}}{\Big(C_{n\sigma,(n+m)\sigma}(\mbf{e}_k)\Big)^2}&=  {\rm e}^ {-2\lambda_1(n+m)\sigma}\Expx{\mbf{x}}{\underset{i\in E}{\sum}\int_{n\sigma}^{(n+m)\sigma}([\mbf{M}((n+m)\sigma-u)\mbf{e}_k]_i)^2{2c_iX_{u,i}}\ud u }.
\end{split}\end{equation*}
Denote by $C=\max\{c_i:i\in E\}$.
Observe that $(\mbf{M}(t)\mbf{e}_k)_i=M(t)_{i,k}$ and by equation \eqref{limit M barczy} exists $C_1>0$ such that $\|\mbf{M}(t)\|\leq C_1 {\rm e}^ {\lambda_1 t}$. Recall that  $m_{\mbf{\phi}}=\min\{\phi(i): i\in E\}$. Then
\begin{equation*}\begin{split}
\Expx{\mbf{x}}{\Big(C_{n\sigma,(n+m)\sigma}(\mbf{e}_k)\Big)^2}&\leq\frac{2CC_1^2}{m_{\mbf{\phi}}}  {\rm e}^ {-2\lambda_1(n+m)\sigma}\Expx{\mbf{x}}{\int_{n\sigma}^{(n+m)\sigma} {\rm e}^ {2\lambda_1[(n+m)\sigma-u]}{\bra{\mbf{\phi}}{\mbf{X}_u}\ud u} }\\
&=\frac{2CC_1^2 \bra{\mbf{\phi}}{\mbf{x}}}{\lambda_1m_{\mbf{\phi}}} [ {\rm e}^ {-\lambda_1n\sigma}- {\rm e}^ {-\lambda_1(n+m)\sigma}]\\
&=\frac{2CC_1^2  \bra{\mbf{\phi}}{\mbf{x}}}{\lambda_1m_{\mbf{\phi}}} [1- {\rm e}^ {-\lambda_1 m\sigma}] {\rm e}^ {-\lambda_1n\sigma},
\end{split}
\end{equation*}
where in the first equality we used the fact that $W_u= {\rm e}^ {-\lambda_1 u}\bra{\mbf{\phi}}{\mbf{X}_u}$ is a martingale. Therefore
\begin{equation}
\label{borel c}
\underset{n=1}{\overset{\infty}{\sum}}\Expx{\mbf{x}}{\Big(C_{n\sigma,(n+m)\sigma}(\mbf{e}_k)\Big)^2}<\infty.
\end{equation}
Then we have the $\mathbb{L}^2(\p_{\mbf{x}})$-convergence. The $\p_{\mbf{x}}$ a.s. convergence follows from Chebyshev's inequality,  Borel-Cantelli Lemma and the previous inequality.

\bigskip

\noindent\underline{\it (ii) Lattice convergence of $S_{t,t+s}(\mbf{e}_k)$}: We will show that, 	if \eqref{xlogx} holds, then for any $k\in E$, $m\in \N$, $\sigma>0$ and $\mbf{x}\in\R^d_+$,
	\begin{equation*}
	\lim_{n\rightarrow \infty}S_{ n\sigma,(n+m)\sigma}(\mbf{e}_k)=0, \qquad \mbox{ in } \mathbb{L}^2(\p_{\mbf{x}}) \mbox{ and } \p_{\mbf{x}}\mbox{ a.s.}
	\end{equation*}

Similar to the proof in (i) above, for $t\in [n\sigma,(n+m)\sigma]$, the process,
\begin{equation}
\label{s n,m}
S_t^{(n,m,\sigma)}:= {\rm e}^ {-\lambda_1(n+m)\sigma}\underset{i\in E}{\sum} \int_{n\sigma}^{t}\int_{\R_+^d}\int_0^{\infty} \bra{\mbf{M}((n+m)\sigma-u)\mbf{e}_k}{\mbf{z}}\Indi{r\leq X_{u-,i}}\widetilde{N}_i^{(1)}(\ud u,\ud \mbf{z}, \ud r)	
\end{equation}
 is a martingale with quadratic variation given by
	$$
 {\rm e}^ {-2\lambda_1(n+m)
\sigma}\underset{i\in E}{\sum} \int_{n\sigma}^{t}\int_{\R_+^d}\int_0^{\infty} \bra{\mbf{M}((n+m)\sigma-u)\mbf{e}_k}{\mbf{z}}^2\Indi{r\leq X_{u,i}}\Indi{\bra{\mbf{1}}{\mbf{z}}\leq  {\rm e}^ {\lambda_1 u}}\ud r\mu_i( \ud\mbf{z}) \ud u.$$
	Then, by taking $t=(n+m)\sigma$, we have
\begin{equation*}\begin{split}
&\Expx{\mbf{x}}{\Big(S_{n\sigma,(n+m)\sigma}(\mbf{e}_k)\Big)^2}\\
&=  {\rm e}^ {-2\lambda_1(n+m)\sigma} \underset{i\in E}{\sum} \Expx{\mbf{x}}{\int_{n\sigma}^{(n+m)\sigma}\int_{\R_+^d}\bra{\mbf{M}((n+m)\sigma-u)\mbf{e}_k}{\mbf{z}}^2 X_{u,i}\Indi{\bra{\mbf{1}}{\mbf{z}}\leq  {\rm e}^ {\lambda_1 u}}\mu_i( \ud\mbf{z}) \ud u}.
\end{split}\end{equation*}
By equation \eqref{limit M barczy}, there exists $C>0$ such that for all $t\geq 0$,
\begin{equation}
\label{equation Mtz}
\bra{\mbf{M}(t)\mbf{e}_k}{\mbf{z}}\leq C {\rm e}^ {\lambda_1 t}\bra{\mbf{1}}{\mbf{z}}.
\end{equation}
 Therefore, by using the definition of $m_{\mbf{\phi}}$
\begin{align*}
&\Expx{\mbf{x}}{\Big(S_{n\sigma,(n+m)\sigma}(\mbf{e}_k)\Big)^2}\\
&\leq \frac{C^2}{m_{\mbf{\phi}}} \underset{i\in E}{\sum} \Expx{\mbf{x}}{\int_{n\sigma}^{(n+m)\sigma}\int_{\R_+^d} {\rm e}^ {-2\lambda_1 u}\bra{\mbf{1}}{\mbf{z}}^2 \bra{\mbf{\phi}}{\mbf{X}_{u}}\Indi{\bra{\mbf{1}}{\mbf{z}}\leq  {\rm e}^ {\lambda_1 u}}\mu_i( \ud\mbf{z}) \ud u}\\
&= \frac{C^2 \bra{\mbf{\phi}}{\mbf{x}}}{m_{\mbf{\phi}}} \underset{i\in E}{\sum} {\int_{n\sigma}^{(n+m)\sigma}\int_{\R_+^d} {\rm e}^ {-\lambda_1 u}\bra{\mbf{1}}{\mbf{z}}^2 \Indi{\bra{\mbf{1}}{\mbf{z}}\leq  {\rm e}^ {\lambda_1 u}}\mu_i( \ud\mbf{z}) \ud u},
\end{align*}
where in the first equality we used the fact that $W_u= {\rm e}^ {-\lambda_1 u}\bra{\mbf{\phi}}{\mbf{X}_u}$ is a martingale. Taking sum over $n$, we get
\begin{align*}
 \underset{n=1}{\overset{\infty}{\sum}}\Expx{\mbf{x}}{\Big(S_{n\sigma,(n+m)\sigma}(\mbf{e}_k)\Big)^2} \leq\frac{C^2 \bra{\mbf{\phi}}{\mbf{x}}}{m_{\mbf{\phi}}} \underset{n=1}{\overset{\infty}{\sum}} \underset{i\in E}{\sum} {\int_{n\sigma}^{\infty}\int_{\R_+^d} {\rm e}^ {-\lambda_1 u}\bra{\mbf{1}}{\mbf{z}}^2 \Indi{\bra{\mbf{1}}{\mbf{z}}\leq  {\rm e}^ {\lambda_1 u}}\mu_i( \ud\mbf{z}) \ud u}.
\end{align*}
By Fubini's Theorem applied to the Lebesgue measure in $\R$ and the countable measure in $\N$, we get
\begin{align*}
\underset{n=1}{\overset{\infty}{\sum}}\Expx{\mbf{x}}{\Big(S_{n\sigma,(n+m)\sigma}(\mbf{e}_k)\Big)^2}&\leq
\frac{C^2 \bra{\mbf{\phi}}{\mbf{x}}}{ m_{\mbf{\phi}}}
\underset{i\in E}{\sum} {\int_{\sigma}^\infty
	\underset{n=1}{\overset{\lfloor u/\sigma\rfloor}{\sum}}\int_{\R_+^d} {\rm e}^ {-\lambda_1 u}\bra{\mbf{1}}{\mbf{z}}^2 \Indi{\bra{\mbf{1}}{\mbf{z}}\leq  {\rm e}^ {\lambda_1 u}}\mu_i( \ud\mbf{z}) \ud u}\\
&\leq
 \frac{C^2 \bra{\mbf{\phi}}{\mbf{x}}}{\sigma m_{\mbf{\phi}}}
 \underset{i\in E}{\sum} {\int_0^\infty\int_{\R_+^d}u {\rm e}^ {-\lambda_1 u}\bra{\mbf{1}}{\mbf{z}}^2 \Indi{\bra{\mbf{1}}{\mbf{z}}\leq  {\rm e}^ {\lambda_1 u}}\mu_i( \ud\mbf{z}) \ud u}.
\end{align*}
By Fubini's Theorem, for each $i\in E$,
\begin{align*}
&\int_{0}^{\infty}\int_{\R_+^d}u {\rm e}^ {-\lambda_1 u}\bra{\mbf{1}}{\mbf{z}}^2\Indi{\bra{\mbf{1}}{\mbf{z}}\leq  {\rm e}^ {\lambda_1 u}}\mu_i( \ud\mbf{z})\ud u\\
&=\int_{\bra{\mbf{1}}{\mbf{z}}\leq 1}\int_{0}^{\infty}u {\rm e}^ {-\lambda_1 u}\bra{\mbf{1}}{\mbf{z}}^2\mu_i( \ud\mbf{z})\ud u+\int_{\bra{\mbf{1}}{\mbf{z}}> 1}\int_{\lambda_1^{-1}\ln(\bra{\mbf{1}}{\mbf{z}})}^{\infty}u {\rm e}^ {-\lambda_1 u}\bra{\mbf{1}}{\mbf{z}}^2\ud u\mu_i( \ud\mbf{z})\\
&=\frac{1}{\lambda_1^2}\left(\int_{\bra{\mbf{1}}{\mbf{z}}\leq 1}\bra{\mbf{1}}{\mbf{z}}^2\mu_i( \ud\mbf{z})+\int_{\bra{\mbf{1}}{\mbf{z}}> 1}\bra{\mbf{1}}{\mbf{z}}(\ln(\bra{\mbf{1}}{\mbf{z}})+1)\mu_i( \ud\mbf{z})\right)<\infty.
\end{align*}
Since \eqref{xlogx} holds,
\begin{equation}
\label{borel s}
 \underset{n=1}{\overset{\infty}{\sum}}\Expx{\mbf{x}}{\Big(S_{n\sigma,(n+m)\sigma}(\mbf{e}_k)\Big)^2}<\infty
\end{equation}
and we have the convergence in $\mathbb{L}^2(\p_{\mbf{x}})$. By Chebyshev's inequality and Borel-Cantelli Lemma we have the $\p_{\mbf{x}}$ a.s. convergence.

\bigskip

\noindent\underline{\it (iii) Lattice convergence of $B_{t,t+s}(\mbf{e}_k)$}:
We show that if \eqref{xlogx} holds, then for any $k\in E$, $m\in \N$, $\sigma>0$ and $\mbf{x}\in\R^d_+$,
	\begin{equation*}
	\lim_{n\rightarrow \infty}B_{ n\sigma,(n+m)\sigma}(\mbf{e}_k)=0, \qquad \mbox{ in } \mathbb{L}^1(\p_{\mbf{x}}) \mbox{ and } \p_{\mbf{x}} \mbox{ a.s.}
	\end{equation*}

Note that for any random measure $N$, we have $\widetilde{N}(A)\leq N(A)+\widehat{N}(A)$, where $\widetilde{N}$ is the compensated measure and $\widehat{N}$ the intensity measure. Then, by inequality \eqref{equation Mtz}

\begin{equation}\label{equation b}
\begin{split}
|B_{ n\sigma,(n+m)\sigma}(\mbf{e}_k)|&\\
&\hspace{-2cm}\leq C\underset{i\in E}{\sum}\int_{n\sigma}^{(n+m)\sigma}\int_{\R_+^d}\int_0^{\infty}  {\rm e}^ {-\lambda u}\bra{\mbf{1}}{\mbf{z}}\Indi{r\leq X_{u,i}}\Indi{\bra{\mbf{1}}{\mbf{z}}>  {\rm e}^ {\lambda_1 u}}(\ud N_i+\ud r\mu_i(\ud \mbf{z})\ud u)\\
&\hspace{-2cm}\leq C\underset{i\in E}{\sum}\int_{n\sigma}^{\infty}\int_{\R_+^d}\int_0^{\infty}  {\rm e}^ {-\lambda u}\bra{\mbf{1}}{\mbf{z}}\Indi{r\leq X_{u,i}}\Indi{\bra{\mbf{1}}{\mbf{z}}>  {\rm e}^ {\lambda_1 u}}(\ud N_i+\ud r\mu_i(\ud \mbf{z})\ud u).
\end{split}
\end{equation}
Using the fact that  $W$ is a martingale, we have
\begin{equation*}
\begin{split}
\Expx{\mbf{x}}{|B_{ n\sigma,(n+m)\sigma}(\mbf{e}_k)|}
&\leq \frac{2C}{m_{\mbf{\phi}}}\underset{i\in E}{\sum}\Expx{\mbf{x}}{\int_{n\sigma}^{\infty}\int_{\R_+^d}  {\rm e}^ {-\lambda u}\bra{\mbf{1}}{\mbf{z}}\bra{\mbf{\phi}}{ X_{u}}\Indi{\bra{\mbf{1}}{\mbf{z}}>  {\rm e}^ {\lambda_1 u}}\mu_i(\ud \mbf{z})\ud u}\\
&=\frac{2C \bra{\mbf{\phi}}{\mbf{x}}}{m_{\mbf{\phi}}}\underset{i\in E}{\sum}\int_{n\sigma}^{\infty}\int_{\R_+^d} \bra{\mbf{1}}{\mbf{z}}
\Indi{\bra{\mbf{1}}{\mbf{z}}>  {\rm e}^ {\lambda_1 u}}\mu_i(\ud \mbf{z})\ud u.
\end{split}.
\end{equation*}
By Fubini's Theorem,
$$\Expx{\mbf{x}}{|B_{ n\sigma,(n+m)\sigma}(\mbf{e}_k)|}\leq \frac{2C \bra{\mbf{\phi}}{\mbf{x}}}{m_{\mbf{\phi}}\lambda_1}\underset{i\in E}{\sum}\int_{\bra{\mbf{1}}{\mbf{z}}>  {\rm e}^ {\lambda_1\sigma n}}\bra{\mbf{1}}{\mbf{z}}\ln(\bra{\mbf{1}}{\mbf{z}})\mu_i(\ud \mbf{z}).$$
Recall that $\int_{1\leq \bra{\mbf{1}}{\mbf{z}}<\infty}\bra{\mbf{1}}{\mbf{z}}\ln(\bra{\mbf{1}}{\mbf{z}})\mu_i(\ud \mbf{z})<\infty$. Therefore
\begin{equation*}\label{limit b}
\underset{n\rightarrow\infty}{\lim}\Expx{\mbf{x}}{|B_{ n\sigma,(n+m)\sigma}(\mbf{e}_k)|}=0,
\end{equation*}
which says $\lim_{n\rightarrow \infty}B_{ n\sigma,(n+m)\sigma}(\mbf{e}_k)=0$ in $\mathbb{L}^1(\p_{\mbf{x}})$.
The $\mathbb{P}_{\mbf{x}}$ a.s. convergence follows from the fact that
$$n\mapsto \int_{n\sigma}^{\infty}\int_{\R_+^d}\int_0^{\infty}  {\rm e}^ {-\lambda u}\bra{\mbf{1}}{\mbf{z}}\Indi{r\leq X_{u,i}}\Indi{\bra{\mbf{1}}{\mbf{z}}>  {\rm e}^ {\lambda_1 u}}(\ud N_i+\ud r\mu_i(\ud \mbf{z})\ud u)$$
is decreasing and inequality \eqref{equation b}.

\bigskip
Now applying (i)-(iii) to \eqref{1-3}, the proof is complete.
\bigskip
\end{proof}

\begin{proposition}\label{lattice case}If \eqref{xlogx} holds, then for any  $\sigma>0$,
\begin{equation*}
\underset{n\rightarrow \infty}{\lim}  {\rm e}^ {-\lambda_1 n\sigma} \mbf{X}_{n\sigma}=W_{\infty}\mbf{\widehat{\phi}}, \qquad \mbox{ in } \mathbb{L}^1(\p_{\mbf{x}}) \mbox{ and } \p_{\mbf{x}}\mbox{ a.s.}\end{equation*}
\end{proposition}

\begin{proof}
Let $k\in E$ and $n,m>0$. By the Markov property, we have
	\begin{equation}\label{condicional final}
\Expconx{\mbf{x}}
{ {\rm e}^ {-\lambda_1(n+m)\sigma}\bra{\mbf{e}_k}{\mbf{X}_{(n+m)\sigma}}}{\mathcal{F}_{n\sigma}}= {\rm e}^ {-\lambda_1(n+m)\sigma}\bra{\mbf{M}(m\sigma)\mbf{e}_k}{\mbf{X}_{n\sigma}}.
	\end{equation}
Let
$$r_m=\left(1-\frac{c_1 {\rm e}^ {-c_2m\sigma}}{\max\{\phi(i)\widehat{\phi}(k): i,k\in E\}}\right)\qquad \mbox{ and } \qquad R_m=\left(1+\frac{c_1 {\rm e}^ {-c_2m\sigma}}{\min\{\phi(i)\widehat{\phi}(k): i,k\in E\}}\right).$$
Observe that $r_m\rightarrow 1$ and $R_m\rightarrow 1$ as $m\rightarrow\infty$. Moreover, by limit \eqref{limit M barczy}, for all $i\in E$
$$r_m\phi(i)\widehat{\phi}(k)\leq  {\rm e}^ {-\lambda_1m\sigma}(\mbf{M}(m\sigma)\mbf{e}_k)_i= {\rm e}^ {-\lambda_1m\sigma}M(m\sigma)_{ik}\leq R_m\phi(i)\widehat{\phi}(k).$$
Hence,
$$r_m\widehat{\phi}(k) {\rm e}^ {-\lambda_1n\sigma}\bra{\mbf{\phi}}{\mbf{X}_{n\sigma}}\leq  {\rm e}^ {-\lambda_1(n+m)\sigma}\bra{\mbf{M}(m\sigma)\mbf{e}_k}{\mbf{X}_{n\sigma}}\leq R_m\widehat{\phi}(k) {\rm e}^ {-\lambda_1n\sigma}\bra{\mbf{\phi}}{\mbf{X}_{n\sigma}}.$$
Now, applying Lemma \ref{lemma conditional}, the previous equation and equation \eqref{condicional final}, we get
\begin{equation*}
\label{veremos}\begin{split}
\limsup_{n\rightarrow \infty} {\rm e}^ {-\lambda_1 n\sigma} X_{n\sigma}(k)=&\limsup_{m\rightarrow \infty}\limsup_{n\rightarrow \infty} {\rm e}^ {-\lambda_1 (n+m)\sigma}\bra{\mbf{e}_k}{\mbf{X}_{(n+m)\sigma}}\\
=&\limsup_{m\rightarrow \infty}\limsup_{n\rightarrow \infty}  {\rm e}^ {-\lambda_1(n+m)\sigma}\bra{\mbf{M}(m\sigma)\mbf{e}_k}{\mbf{X}_{n\sigma}}\\
\leq&\lim_{m\rightarrow \infty}\lim_{n\rightarrow \infty}R_m\widehat{\phi}(k)W_{n\sigma}\\
=&\widehat{\phi}(k)W_{\infty},\quad \p_{\mbf{x}}\mbox{ a.s.}
\end{split}
\end{equation*}
In a similar way,
$$\liminf_{n\rightarrow \infty} {\rm e}^ {-\lambda_1 n\sigma} X_{n\sigma}(k)=\widehat{\phi}(k)W_{\infty}, \qquad \p_{\mbf{x}}\mbox{ a.s.}$$
Therefore, the a.s. assertion is true.
Recall that $m_{\mbf{\phi}}:=\min\{\phi(k), k\in E\}>0$.
Observe that $0\leq {\rm e}^ {-\lambda_1 n\sigma} X_{n\sigma}(k)\leq W_{n\sigma}
m_{\mbf{\phi}}^{-1}
$ and by Theorem \ref{xlogx theorem} the martingale $W_t$ converges in  $\mathbb{L}^1(\p_{\mbf{x}})$.  Then by the Generalized Dominated Convergence Theorem, the $\mathbb{L}^1(\p_{\mbf{x}})$ assertion holds. (see for instance \cite[Problem 12, p. 133]{dudley}),
\end{proof}

\subsection{From lattice times  to continuous times.}

In this section, we extend the convergence along lattice times in Theorem \ref{tstrong1} to convergence along continuous times and conclude our main results.

\begin{proof}[Proof of Theorem \ref{tstrong1}]
Since $\mbf{\phi}$ is a positive vector, then for any $\mbf{a}\in\R^d$ we have $|\bra{\mbf{a}}{\mbf{X}_t}|\leq \frac{\|\mbf{a}\|}{m_{\mbf{\phi}}}\bra{\mbf{\phi}}{\mbf{X}_t}$, where $m_{\mbf{\phi}}:=\min\{\phi(i):i\in E\}$.
Therefore $\P_{\mbf{x}}$- a.s.
\begin{equation*}\begin{split}
&\underset{\sigma\rightarrow 0}{\lim}\ \underset{n\rightarrow \infty}{\lim}\ \underset{t\in [n\sigma,(n+1)\sigma]}{\sup}| {\rm e}^ {-\lambda_1 t}\bra{\mbf{M}((n+1)\sigma-t)\mbf{e}_k}{\mbf{X}_t}- {\rm e}^ {-\lambda_1 t}\bra{\mbf{e}_k}{\mbf{X}_t}|\\
\leq &\underset{\sigma\rightarrow 0}{\lim}\ \underset{n\rightarrow \infty}{\lim}\ \underset{t\in [n\sigma,(n+1)\sigma]}{\sup} {\rm e}^ {-\lambda_1 t}\bra{\mbf{\phi}}{\mbf{X}_t}\frac{\|\mbf{M}((n+1)\sigma-t)\mbf{e}_k-\mbf{e}_k\|}{m_{\mbf{\phi}}}\\
\leq&\left(\underset{\sigma\rightarrow 0}{\lim}\ \underset{n\rightarrow \infty}{\lim}\ \underset{t\in [n\sigma,(n+1)\sigma]}{\sup}W_t\right)\left(\underset{\sigma\rightarrow 0}{\lim}\ \underset{u\in [0,\sigma]}{\sup}\frac{\|\mbf{M}(u)\mbf{e}_k-\mbf{e}_k\|}{m_{\mbf{\phi}}}\right)=0.
\end{split}
\end{equation*}
So, in order to have our result, it is enough to prove that
\begin{equation*}
\underset{\sigma\rightarrow 0}{\lim}\ \underset{n\rightarrow \infty}{\lim}\ \underset{t\in [n\sigma,(n+1)\sigma]}{\sup} {\rm e}^ {-\lambda_1 t}\bra{\mbf{M}((n+1)\sigma-t)\mbf{e}_k}{\mbf{X}_t}=\widehat{\phi}(k)W_\infty\qquad \P_{\mbf{x}}\mbox{ a.s.}
\end{equation*}
By applying equation \eqref{completly useful equation} to $\mbf{e}_k$ and
$n\sigma\le t\le (n+1)\sigma$, we obtain
\begin{align}
& {\rm e}^ {-\lambda_1 t}\bra{\mbf{M}((n+1)\sigma-t)\mbf{e}_k}{\mbf{X}_t}\notag\\
=& {\rm e}^ {-\lambda_1 t}\bra{\mbf{M}(\sigma)\mbf{e}_k}{\mbf{X}_{n\sigma}}\notag\\
&+ {\rm e}^ {-\lambda_1 t}\underset{i\in E}{\sum} \int_{n\sigma}^t[\mbf{M}((n+1)\sigma-u)\mbf{e}_k]_i\sqrt{2c_iX_{u,i}}\ud W_{u,i}\notag\\
&+ {\rm e}^ {-\lambda_1 t}\underset{i\in E}{\sum}\int_{n\sigma}^t\int_{\R_+^d}\int_0^{\infty} \bra{\mbf{M}((n+1)\sigma-u)\mbf{e}_k}{\mbf{z}}\Indi{l\leq X_{u-,i}}\widetilde{N}_i^{(1)}(\ud u,\ud \mbf{z}, \ud l)\notag\\
&+ {\rm e}^ {-\lambda_1 t}\underset{i\in E}{\sum}\int_{n\sigma}^t\int_{\R_+^d}\int_0^{\infty} \bra{\mbf{M}((n+1)\sigma-u)\mbf{e}_k}{\mbf{z}}\Indi{l\leq X_{u-,i}}\widetilde{N}_i^{(2)}(\ud u,\ud \mbf{z}, \ud l)\notag\\
=:& {\rm e}^ {-\lambda_1 t}\bra{\mbf{M}(\sigma)\mbf{e}_k}{\mbf{X}_{n\sigma}}+C_{n,t}^{\sigma}(\mbf{e}_k)+S_{n,t}^{\sigma}(\mbf{e}_k)+B_{n,t}^{\sigma}(\mbf{e}_k).
\label{i-iii}
\end{align}
By the result for lattice times (Proposition \ref{lattice case}), $\P_{\mbf{x}}$ a.s.
$$\underset{\sigma\rightarrow 0}{\lim}\ \underset{n\rightarrow \infty}{\lim}\ \underset{t\in [n\sigma,(n+1)\sigma]}{\sup} {\rm e}^ {-\lambda_1 t}\bra{\mbf{M}(\sigma)\mbf{e}_k}{\mbf{X}_{n\sigma}}=\underset{\sigma\rightarrow 0}{\lim}\ \underset{n\rightarrow \infty}{\lim} {\rm e}^ {-\lambda_1\sigma n}\bra{\mbf{M}(\sigma)\mbf{e}_k}{\mbf{X}_{n\sigma}}=\widehat{\phi}(k)W_\infty.$$
Hence, to complete the proof, we only need to prove that the last three terms on the right-hand side of \eqref{i-iii} converge uniformly for $t\in [n\sigma,(n+1)\sigma]$
first as $n\to\infty$ and then $\sigma\to 0$.
\bigskip

\noindent \underline{\it  (i) Convergence of $C_{n,t}^{\sigma}(\mbf{e}_k)$}: We show that for any $k\in E$ and $\mbf{x}\in\R^d_+$,
	\begin{equation*}
	\underset{\sigma\rightarrow 0}{\lim}\ \underset{n\rightarrow \infty}{\lim}\ \underset{t\in [n\sigma,(n+1)\sigma]}{\sup} C_{n,t}^{\sigma}(\mbf{e}_k)=0 \qquad
\p_{\mbf{x}}\mbox{ a.s.}
	\end{equation*}

	Recall the definition of the martingale $C_t^{(n,1,\sigma)}$ given by \eqref{c n,m}. And note that $$|C_{n,t}^{\sigma}(\mbf{e}_k)|\leq  {\rm e}^ {\lambda_1 \sigma}|C_t^{(n,1,\sigma)}|$$
	Then, by the maximal inequality for martingales, for all $\epsilon>0$ we have that for
	\begin{equation*}
	\Prox{\mbf{x}}{\underset{t\in [n\sigma,(n+1)\sigma]}{\sup} |C_{n,t}^{\sigma}(\mbf{e}_k)|>\epsilon}\leq 	\Prox{\mbf{x}}{\underset{t\in [n\sigma,(n+1)\sigma]}{\sup}  {\rm e}^ {\lambda_1 \sigma}|C_t^{(n,1,\sigma)}|>\epsilon}
	\leq \frac{ {\rm e}^ {2\lambda_1 \sigma}}{\epsilon^2}\Expx{\mbf{x}}{\left|C_{(n+1)\sigma}^{(n,1,\sigma)}\right|^2}
	\end{equation*}
	Note that $C_{(n+1)\sigma}^{(n,1,\sigma)}=C_{n\sigma,(n+1)\sigma}(\mbf{e}_k)$, therefore by \eqref{borel c},
\begin{equation*}
\underset{n=1}{\overset{\infty}{\sum}}	\Prox{\mbf{x}}{\underset{t\in [n\sigma,(n+1)\sigma]}{\sup} |C_{n,t}^{\sigma}(\mbf{e}_k)|>\epsilon}<\infty.
\end{equation*}
By Borel-Cantelli we have the result.
\bigskip

\noindent\underline{\it (ii) Convergence of $S_{n,t}^{\sigma}(\mbf{e}_k)$}: We show that, if \eqref{xlogx} holds,	then for any $k\in E$ and $\mbf{x}\in\R^d_+$,
	\begin{equation*}
	\underset{\sigma\rightarrow 0}{\lim}\ \underset{n\rightarrow \infty}{\lim}\ \underset{t\in [n\sigma,(n+1)\sigma]}{\sup} S_{n,t}^{\sigma}(\mbf{e}_k)=0 \qquad
	\p_{\mbf{x}}\mbox{ a.s.}
	\end{equation*}

The proof is analogous to the previous one. But this time we use the martingale $S_t^{(n,1,\sigma)}$ given by \eqref{s n,m} and equation \eqref{borel s}.

\bigskip

\noindent\underline{\it (iii) Convergence of $B_{n,t}^{\sigma}(\mbf{e}_k)$}: We show that, if \eqref{xlogx} holds, then for any $k\in E$ and $\mbf{x}\in\R^d_+$,
	\begin{equation*}
	\underset{\sigma\rightarrow 0}{\lim}\ \underset{n\rightarrow \infty}{\lim}\ \underset{t\in [n\sigma,(n+1)\sigma]}{\sup} B_{n,t}^{\sigma}(\mbf{e}_k)=0 \qquad
	\p_{\mbf{x}}\mbox{ a.s.}
	\end{equation*}

By inequality \eqref{equation Mtz},
\begin{equation*}
\begin{split}
&|B_{n,t}^{\sigma}(\mbf{e}_k)|\\
&\leq C {\rm e}^ {-\lambda_1(t-(n+1)\sigma)}\underset{i\in E}{\sum}\int_{n\sigma}^{t}\int_{\R_+^d}\int_0^{\infty}  {\rm e}^ {-\lambda u}\bra{\mbf{1}}{\mbf{z}}\Indi{r\leq X_{u,i}}\Indi{\bra{\mbf{1}}{\mbf{z}}>  {\rm e}^ {\lambda_1 u}}(\ud N_i+\ud r\mu_i(\ud \mbf{z})\ud u)\\
&\leq C {\rm e}^ {-\lambda_1(t-(n+1)\sigma)}\underset{i\in E}{\sum}\int_{n\sigma}^{\infty}\int_{\R_+^d}\int_0^{\infty}  {\rm e}^ {-\lambda u}\bra{\mbf{1}}{\mbf{z}}\Indi{r\leq X_{u,i}}\Indi{\bra{\mbf{1}}{\mbf{z}}>  {\rm e}^ {\lambda_1 u}}(\ud N_i+\ud r\mu_i(\ud \mbf{z})\ud u).
\end{split}
\end{equation*}
Then,
\begin{equation*}
\begin{split}
\underset{t\in [n\sigma,(n+1)\sigma]}{\sup}|B_{n,t}^{\sigma}(\mbf{e}_k)|&\\
&\hspace{-2cm}\leq C {\rm e}^ {\lambda_1\sigma}\underset{i\in E}{\sum}\int_{n\sigma}^{\infty}\int_{\R_+^d}\int_0^{\infty}  {\rm e}^ {-\lambda u}\bra{\mbf{1}}{\mbf{z}}\Indi{r\leq X_{u,i}}\Indi{\bra{\mbf{1}}{\mbf{z}}>  {\rm e}^ {\lambda_1 u}}(\ud N_i+\ud r\mu_i(\ud \mbf{z})\ud u).
\end{split}
\end{equation*}
The claim is true by following the same steps after equation \eqref{equation b}.
We omit the details here.

\bigskip

 Putting the above three conclusions together, we now conclude
$$\underset{t\rightarrow\infty}{\lim} {\rm e}^ {-\lambda_1 t}\bra{\mbf{e}_k}{\mbf{X}_{t}}=\widehat{\phi}(k)W_\infty,\qquad \P_{\mbf{x}}\ \mbox{a.s.}$$
as required.\end{proof}

\section*{Acknowledgements}
A.K.
and S.P. would like to acknowledge support from a Royal Society Newton International Fellowship held by S.P. and the Advanced Newton Fellowship held collaboratively between A.K. and  J. C. Pardo. 
Y.-X.R. would like to acknowledge support from NSFC (Grant No.  11671017 and 11731009), and LMEQF.

\bibliographystyle{alea3}
\bibliography{bibliography}

\begin{thebibliography}{11}
\providecommand{\natexlab}[1]{#1}
\providecommand{\url}[1]{\texttt{#1}}
\providecommand{\urlprefix}{URL }
\expandafter\ifx\csname urlstyle\endcsname\relax
  \providecommand{\doi}[1]{doi:\discretionary{}{}{}#1}\else
  \providecommand{\doi}{doi:\discretionary{}{}{}\begingroup
  \urlstyle{rm}\Url}\fi
\providecommand{\eprint}[2][]{\url{#2}}

\bibitem[{Athreya and Ney(2004)}]{an}
K.~B. Athreya and P.~E. Ney.
\newblock \emph{Branching processes}.
\newblock Dover Publications, Inc., Mineola, NY (2004).

\bibitem[{Barczy et~al.(2015)Barczy, Li and Pap}]{blp1}
M.~Barczy, Z.~Li and G.~Pap.
\newblock Stochastic differential equation with jumps for multi-type continuous
  state and continuous time branching processes with immigration.
\newblock \emph{ALEA Lat. Am. J. Probab. Math. Stat.} \textbf{12}~(1), 129--169
  (2015).

\bibitem[{Barczy and Pap(2016)}]{bp2005}
M.~Barczy and G.~Pap.
\newblock Asymptotic behavior of critical, irreducible multi-type continuous
  state and continuous time branching processes with immigration.
\newblock \emph{Stoch. Dyn.} \textbf{16}~(4) (2016).

\bibitem[{Chen et~al.(2017)Chen, Ren and Song}]{CRS}
Z.-Q. Chen, Y.-X. Ren and R.~Song.
\newblock ${L}\log {L}$ criterion for a class of multitype superdiffusions with
  nonlocal branching mechanism.
\newblock \emph{arXiv preprint 1708.08219}  (2017).

\bibitem[{Dudley(2002)}]{dudley}
R.~M. Dudley.
\newblock \emph{Real analysis and probability}, volume~74 of \emph{Cambridge
  Studies in Advanced Mathematics}.
\newblock Cambridge University Press, Cambridge (2002).

\bibitem[{Dynkin(1991)}]{dyn91}
E.~B. Dynkin.
\newblock Branching particle systems and superprocesses.
\newblock \emph{Ann. Probab.} \textbf{19}~(3), 1157--1194 (1991).

\bibitem[{Dynkin(2002)}]{dynkin}
E.~B. Dynkin.
\newblock \emph{Diffusions, superdiffusions and partial differential
  equations}.
\newblock American Mathematical Society, Providence, RI (2002).

\bibitem[{Dynkin and Kuznetsov(2004)}]{DK04}
E.~B. Dynkin and S.~E. Kuznetsov.
\newblock {$\N$}-measures for branching exit {M}arkov systems and their
  applications to differential equations.
\newblock \emph{Probab. Theory Related Fields} \textbf{130}~(1), 135--150
  (2004).

\bibitem[{Kyprianou and Palau(2017)}]{kp}
A.~Kyprianou and S.~Palau.
\newblock Extinction properties of multi-type continuous-state branching
  processes.
\newblock \emph{Stochastic Processes and their Applications}  (2017).

\bibitem[{Li(2011)}]{Li}
Zenghu Li.
\newblock \emph{Measure-valued branching {M}arkov processes}.
\newblock Probability and its Applications (New York). Springer, Heidelberg
  (2011).

\bibitem[{{Ren} et~al.(2016){Ren}, {Song} and {Yang}}]{rensongyang}
Y.-X. {Ren}, R.~{Song} and T.~{Yang}.
\newblock {Spine decomposition and $L\log L$ criterion for superprocesses with
  non-local branching mechanisms}.
\newblock \emph{ArXiv preprint 1609.02257}  (2016).

\end{thebibliography}

\end{document}